\RequirePackage[l2tabu, orthodox]{nag} % In­hibit use of non-ams­math math­e­mat­ics markup when us­ing ams­math
\documentclass[11pt, a4paper]{article}

\title{Implementation of Interior-point Methods for LP\\based on Krylov Subspace Iterative Solvers\\with Inner-iteration Preconditioning}
\author{Yiran Cui\thanks{Department of Computer Science, University College London, Gower Street, London WC1E 6BT, United Kingdom. \texttt{y.cui.12@ucl.ac.uk}} \and Keiichi Morikuni\thanks{Division of Information Engineering, University of Tsukuba, Tenoudai 1-1-1, Tsukuba, Ibaraki 305-8573, Japan. The author was supported in part by JSPS KAKENHI Grant Number~16K17639.  \texttt{morikuni@cs.tsukuba.ac.jp}} \and Takashi Tsuchiya\thanks{National Graduate Institute for Policy Studies, 7-22-1 Roppongi, Minato, Tokyo 106-8677, Japan. The author was supported in part by JSPS KAKENHI Grant Number 15H02968. \texttt{tsuchiya@grips.ac.jp}} \and Ken Hayami\thanks{National Institute of Informatics, SOKENDAI (The Graduate University for Advanced Studies), 2-1-2 Hitotsubashi, Chiyoda, Tokyo 101-0003, Japan. The author was supported in part by JSPS KAKENHI Grant Number 15K04768 and 15H02968. \texttt{hayami@nii.ac.jp}}}

\date{}

\usepackage{graphicx}
\usepackage{afterpage} % to allow landscape table float
\usepackage{algorithm, algorithmic}
\usepackage{amsmath, amsfonts, amssymb, amsthm}
\usepackage{color}
\usepackage{multirow}
\usepackage{pdflscape, rotating} % for horizontal table
\usepackage{placeins}
\usepackage[caption=false]{subfig}\usepackage[caption=false]{subfig}
\makeatletter \let\cl@part\relax \makeatother
\usepackage{mathtools} % floor and ceil

\DeclarePairedDelimiter\floor{\lfloor}{\rfloor}

\usepackage[all, warning]{onlyamsmath} % De­tect­ing and warn­ing about ob­so­lete LaTeX com­mands
\usepackage[
	bookmarks=true, 
	bookmarksnumbered=true, 
	bookmarkstype=toc, 
	colorlinks={true},%
	pdfauthor={Yiran Cui, Keiichi Morikuni, Takashi Tsuchiya, Ken Hayami},%
	pdfdisplaydoctitle={true},
	pdfkeywords={Linear programming problems Interior-point methods Inner-iteration preconditioning Krylov subspace methods},%
	pdfsubject={},
	pdftitle={Implementation of Interior-point Methods for LP based on Krylov Subspace Iterative Solvers with Inner-iteration Preconditioning},	
	setpagesize={false},%
	urlcolor={red} %,
]{hyperref}
\usepackage[T1]{fontenc}
\usepackage{exscale, fix-cm, fullpage, lmodern, textcomp}

\hypersetup{pdfpagemode=UseNone} % don't show bookmarks on initial view

\theoremstyle{plain}
\newtheorem{theorem}{Theorem}[section]
\newtheorem{lemma}[theorem]{Lemma}
\newtheorem{remark}[theorem]{Remark}
\renewcommand{\arraystretch}{0.94}
\renewcommand{\tabcolsep}{0.8mm}
\newfont{\bg}{cmr9 scaled\magstep4}

\DeclareMathOperator{\bigzerol}{\lower1.0ex\hbox{\bg 0}}

\begin{document}
\maketitle

\begin{abstract}
We apply novel inner-iteration preconditioned Krylov subspace \break methods to the interior-point algorithm for linear programming (LP). Inner-iteration preconditioners recently proposed by Morikuni and Hayami enable us to overcome the severe ill-conditioning of linear equations solved in the final phase of interior-point iterations. The Krylov subspace methods do not suffer from rank-deficiency and therefore no preprocessing is necessary even if rows of the constraint matrix are not linearly independent. 
By means of these methods, a new interior-point recurrence is proposed in order to omit one matrix-vector product at each step.
Extensive numerical experiments are conducted over diverse instances of 140 LP problems including the Netlib, QAPLIB, Mittelmann and Atomizer Basis Pursuit collections.
The largest problem has 434,580 unknowns.
It turns out that our implementation is more robust than the standard public domain solvers SeDuMi (Self-Dual Minimization), SDPT3 (Semidefinite Programming Toh-Todd-T\"{u}t\"{u}nc\"{u}) and the LSMR iterative solver in PDCO (Primal-Dual Barrier Method for Convex Objectives) without increasing CPU time.
The proposed interior-point method based on iterative solvers succeeds in solving a fairly large number of LP instances from benchmark libraries under the standard stopping criteria.
The work also presents a fairly extensive benchmark test for several renowned solvers including direct and iterative solvers.
\end{abstract}

\section{Introduction}
Consider the linear programming (LP) problem in the standard primal-dual formulation
\begin{subequations}\label{eq:PrimalDualPair}
\begin{align}
\min_{\boldsymbol{x}} \boldsymbol{c}^\mathsf{T}\boldsymbol{x} \qquad & \mbox{subject to} \qquad A\boldsymbol{x}=\boldsymbol{b}, \;\; \boldsymbol{x}\geq \boldsymbol{0},\\
\max_{\boldsymbol{y},\boldsymbol{s}} \boldsymbol{b}^\mathsf{T}\boldsymbol{y} \qquad &\mbox{subject to} \qquad  A^\mathsf{T}\boldsymbol{y} + \boldsymbol{s}=\boldsymbol{c},\;\; \boldsymbol{s}\geq \boldsymbol{0},
\end{align}
\end{subequations}
where $A \in \mathbb{R}^{m\times n}$, $m\leq n$, and we assume the existence of an optimal solution.
In this paper, we describe an implementation of the interior-point method for LP based on iterative solvers.  
The main computational task in one iteration of the interior-point method is the solution of a system of linear equations to compute the search direction. 

For this task, direct solvers are usually used. 
But some solvers also employ iterative solvers.
Iterative solvers are advantageous when the systems are large and sparse, or even when they are large and dense but the product of the coefficient matrix and a vector can be approximated cheaply, as in \cite{ChenDonohoSaunders1998,saunders2002pdco}. The difficulty with iterative solvers is that the linear system becomes notoriously ill-conditioned towards the end of interior-point iterations. 
One approach is to precondition the mathematically equivalent indefinite augmented system (as in equation \eqref{eq:SE2}) as in HOPDM (Higher Order Primal-Dual Method) \cite{gondzio1995hopdm} and also \cite{ChinVannelli1994,FreundJarre1997,FreundJarreMizuno1999,BergamaschiGondzioZilli2004,OliveiraSorensen2005,BergamaschiGondzioVenturinZilli,rees2007preconditioner,al2008preconditioning,al2009convergence,Gondzio2012a}.
The other approach is to precondition the equivalent normal equations (as in equation \eqref{eq:SE3})
\cite{GillMurraySaundersTomlinWright1986,KarmarkarRamakrishnan1991,Mehrotra1992b,CarpenterShanno1993,LustigMarstenShanno1994,MehrotraWang1996,PortugalResendeVeigaJudice2000,Korzak2000,wang2000adaptive,Cui2009th}.

In this paper, we treat the normal equations and apply novel inner-iteration preconditioned Krylov subspace methods to them.
The inner-iteration preconditioners recently proposed by Morikuni and Hayami \cite{MorikuniHayami2013,MorikuniHayami2015} enable us to deal with the severe ill-conditioning of the normal equations. 
Furthermore, the proposed Krylov subspace methods do not suffer from singularity and therefore no preprocessing is necessary even if $A$ is rank-deficient. 

The main contribution of the present paper is that we actually show that the use of the inner-iteration preconditioner enables the efficient interior-point solution of wide-ranging LP problems. We further proposed combining the row-scaling scheme with the inner-outer iteration methods, where the row norm appears in the successive overrelaxation (SOR) inner-iterations, to improve the condition of the system at each interior-point step. The linear systems are solved with a gradually tightened stopping tolerance. We proposed a new recurrence in order to omit one matrix-vector product at each interior-point step. These techniques reduce the CPU time.

Extensive numerical experiments were conducted over diverse instances of 127 LP problems taken from the standard benchmark libraries Netlib, QAPLIB, and Mittelmann collections. 
The largest problem has 434,580 unknowns. 
The proposed interior-point method is entirely based on iterative solvers and yet succeeds in solving a fairly large number of standard LP instances from the benchmark libraries with standard stopping criteria.
We could not find any other analogous result where this level of LP instances were solved just relying on iterative solvers.

We compared our interior-point LP solvers based on AB-GMRES (right-preconditioned generalized minimal residual method) \cite{HayamiYinIto2010,MorikuniHayami2015}, CGNE, and MRNE (preconditioned CG and MINRES applied to the normal equations of the second kind) \cite{Craig1955,MorikuniHayami2015} with the following well-known interior-point LP solvers:
\begin{enumerate}
\item SeDuMi (Self-Dual Minimization) \cite{Sturm1999}, (public-domain, direct solver),
\item SDPT3 (Semidefinite Programming Toh-Todd-T\"{u}t\"{u}nc\"{u}) \cite{TohToddTutuncu1999,TutuncuTohTodd2003}, (public-domain, direct solver),
%\item PDCO (Primal-Dual Barrier Method for Convex Objectives) \cite{saunders2002pdco},
\begin{enumerate}
	\renewcommand\labelenumi{\theenumi}
\item PDCO-Direct (public-domain, direct solver),
\item PDCO-LSMR (public-domain, LSMR iterative solver),
\end{enumerate}
\item MOSEK \cite{mosek} (commercial, direct solver).
\end{enumerate}

SeDuMi and SDPT3 are solvers for conic linear programming including semidefinite programming (SDP) and second-order cone programming (SOCP). PDCO is for LP and convex quadratic programming (QP) and has options to solve the system of linear equations with Krylov subspace iterative method LSMR in addition to the direct method.  MOSEK is considered as one of the state-of-the-art solvers for LP.  

As summarized in Table \ref{tab:overall}, our implementation was able to solve most
instances, which is clearly superior to SeDuMi, SDPT3, PDCO-Direct, and PDCO-LSMR 
with comparable computation time, though it is still slower than MOSEK.

We also tested our solvers on different problems which arise in basis pursuit \cite{ChenDonohoSaunders1998} where the coefficient matrix is much denser than the aforementioned standard benchmark problems. 	

We emphasize that there are many interesting topics to be further worked out based on this paper.
There is still room for improvement regarding the iterative solvers as well as using more sophisticated methods for the interior-point iterations.

In the following, we introduce the interior-point method and review the iterative solvers previously used.
We employ an infeasible primal-dual predictor-corrector interior-point method, one of the methods that evolved from the original primal-dual interior-point method 
\cite{Tanabe1988,KojimaMizunoYoshise1989,MonteiroAdler1989,Wright1997} incorporating several innovative ideas, e.g.,~\cite{Zhang1994,Mehrotra1992b}. 
 
An optimal solution $\boldsymbol{x}, \boldsymbol{y}, \boldsymbol{s}$ to problem \eqref{eq:PrimalDualPair} must satisfy the Karush-Kuhn-Tucker (KKT) conditions
\begin{subequations}
\label{eq:KKT}
\begin{align}
	A^\mathsf{T} \boldsymbol{y} + \boldsymbol{s} & = \boldsymbol{c}, \label{eq:KKT1} \\	
	A \boldsymbol{x} & = \boldsymbol{b}, \label{eq:KKT2} \\
	X S \boldsymbol{e} & = \boldsymbol{0}, \label{eq:KKT3} \\
	\boldsymbol{x} \geq\boldsymbol{0}, \quad \boldsymbol{s} & \geq \boldsymbol{0}, \label{eq:KKT4}
\end{align}
\end{subequations}
where $X :=\mathrm{diag}(x_1, x_2, \dots, x_n)$, $S:=\mathrm{diag}(s_1, s_2, \dots, s_n)$, and $\boldsymbol{e} := [1,1,\dots, 1]^\mathsf{T}$. 
The complementarity condition \eqref{eq:KKT3} implies that at an optimal solution, one of the elements $x_i$ or $s_i$ must be zero for $i = 1, 2, \dots, n$.

The following system is obtained by relaxing \eqref{eq:KKT3} to $XS{\boldsymbol e}= \mu {\boldsymbol e}$ with $\mu >0$:
  \begin{equation}\label{eq:Relaxed_KKT}
    XS\boldsymbol{e} = \mu \boldsymbol{e}, \ \ A\boldsymbol{x} = \boldsymbol{b}, \ \ A^\mathsf{T} \boldsymbol{y}+\boldsymbol{s}=\boldsymbol{c}, \ \ \boldsymbol{x}\geq \boldsymbol{0},\ \ \boldsymbol{s} \geq \boldsymbol{0}.
  \end{equation}
The interior-point method solves the problem \eqref{eq:PrimalDualPair} by generating solutions to \eqref{eq:Relaxed_KKT}, with $\mu$ decreasing towards zero, so that \eqref{eq:KKT} is satisfied within some tolerance level at the solution point.
The search direction at each infeasible interior-point step is obtained by solving the Newton equations

\begin{equation}\label{eq:SE1}
\begin{bmatrix} \boldsymbol{0} & A^\mathsf{T} & I \\ A & \boldsymbol{0} & \boldsymbol{0} \\ S & \boldsymbol{0} & X \end{bmatrix}  \left[ \begin{array}{c} \Delta\boldsymbol{x} \\ \Delta\boldsymbol{y} \\ \Delta\boldsymbol{s} \end{array} \right]= \left[\begin{array}{c} \boldsymbol{r}_\mathrm{d} \\ \boldsymbol{r}_\mathrm{p} \\
\boldsymbol{r}_\mathrm{c} 
 \end{array} \right],
\end{equation}
where $\boldsymbol{r}_\mathrm{d} := \boldsymbol{c} - A^\mathsf{T} \boldsymbol{y} - \boldsymbol{s} \in \mathbb{R}^n$ is the residual of the dual problem, $\boldsymbol{r}_\mathrm{p} := \boldsymbol{b} - A\boldsymbol{x}  \in \mathbb{R}^m$ is the residual of the primal problem,  $\boldsymbol{r}_\mathrm{c} := -XS\boldsymbol{e} +  \sigma\mu \boldsymbol{e}$
, $\mu := {\boldsymbol{x}^\mathsf{T}\boldsymbol{s}}/{n}$ is the duality measure, and $\sigma \in [0, 1)$ is the centering parameter, which is dynamically chosen to govern the progress of the interior-point method. 
Once the $k$th iterate $(\boldsymbol{x}^{(k)}, \boldsymbol{y}^{(k)}, \boldsymbol{s}^{(k)})$ is given and \eqref{eq:SE1} is solved, we define the next iterate as $(\boldsymbol{x}^{(k+1)}, \boldsymbol{y}^{(k+1)}, \boldsymbol{s}^{(k+1)}) := (\boldsymbol{x}^{(k)}, \boldsymbol{y}^{(k)}, \boldsymbol{s}^{(k)}) + \alpha (\Delta\boldsymbol{x}, \Delta\boldsymbol{y}, \Delta\boldsymbol{s})$, where $\alpha \in (0, 1]$ is a step length to ensure the positivity of $\boldsymbol{x}$ and $\boldsymbol{s}$, and then reduce $\mu$ to $\sigma\mu$ before solving \eqref{eq:SE1} again.

At each iteration, the solution of \eqref{eq:SE1} dominates the total CPU time.
The choice of linear solvers depends on the way of arranging the matrix of \eqref{eq:SE1}. 
Aside from solving the $(m+2n)\times(m+2n)$ system \eqref{eq:SE1}, one can solve its reduced equivalent form of size  $(m+n)\times(m+n)$ 
\begin{equation}\label{eq:SE2}
\begin{bmatrix} A & 0\\ S & -XA^\mathsf{T} \end{bmatrix}  \left[ \begin{array}{c} \Delta\boldsymbol{x} \\ \Delta\boldsymbol{y} \end{array}\right]= \left[\begin{array}{c} \boldsymbol{r}_\mathrm{p} \\ \boldsymbol{r}_\mathrm{c} - X\boldsymbol{r}_\mathrm{d}\end{array} \right],
\end{equation}
or a more condensed equivalent form of size $m\times m$
\begin{equation}\label{eq:SE3}
AXS^{-1}A^\mathsf{T}\Delta \boldsymbol{y} = \boldsymbol{r}_\mathrm{p} - AS^{-1}(\boldsymbol{r}_\mathrm{c} - X\boldsymbol{r}_\mathrm{d}),
\end{equation}
both of which are obtained by performing block Gaussian eliminations on \eqref{eq:SE1}. 
We are concerned in this paper with solving the third equivalent form \eqref{eq:SE3}.

It is known that the matrix of \eqref{eq:SE3} is semidefinite when any of the following cases is encountered. 
First, when $A$ is rank-deficient, system \eqref{eq:SE3} is singular. 
There exist presolving techniques that address this problem, see, e.g.,~\cite{AndersenAndersen1995,Gondzio1997}. 
However, they do not guarantee to detect all dependent rows in $A$.
Second, in late interior-point iterations, the diagonal matrix $XS^{-1}$ has very tiny and very large diagonal values as a result of convergence.
Thus, the matrix may become positive semidefinite.
In particular, the situation becomes severe when primal degeneracy occurs at an optimal solution.
One can refer to \cite{GondzioTerlaky1996,Zhang1998} for more detailed explanations. 
 
Thus, when direct methods such as Cholesky decomposition are applied to \eqref{eq:SE3}, some diagonal pivots encountered during decomposition can be zero or negative, causing the algorithm to break down. 
Many direct methods adopt a strategy of replacing the problematic pivot with a very large number. 
See, e.g.,~\cite{Zhang1998} for the Cholesky-Infinity factorization, which is specially designed to solve \eqref{eq:SE3} when it is positive semidefinite but not definite. 
Numerical experience \cite{AdlerResendeVeigaKarmarkar1989,LustigMarstenShanno1992,FourerMehrotra1993,LustigMarstenShanno1994,AndersenGondzioMeszarosXu1996,Wright1999,CzyzykMehrotraWagnerWright1999} indicates that direct methods provide sufficiently accurate solutions for interior-point methods to converge regardless of the ill-conditioning of the matrix. 
However, as the LP problems become larger, the significant fill-ins in decompositions make direct methods prohibitively expensive. 
It is stated in \cite{Gondzio2012} that the fill-ins are observed even for very sparse matrices. 
Moreover, the matrix can be dense, as in QP in support vector machine training \cite{FerrisMunson2002} or linear programming in basis pursuit \cite{ChenDonohoSaunders1998}, and even when $A$ is sparse, $AXS^{-1}A^\mathsf{T}$ can be dense or have a pattern of nonzero elements that renders the system difficult for direct methods.
The expensive solution of the KKT systems is a usual disadvantage of second-order methods including interior-point methods. 

These drawbacks of direct methods and the progress in preconditioning techniques motivate researchers to develop stable iterative methods for solving \eqref{eq:SE3} or alternatively \eqref{eq:SE2}. 
The major problem is that as the interior-point iterations proceed, the condition number of the term $XS^{-1}$ increases, making the system of linear equations intractable. 
One way to deal with this is to employ suitable preconditioners.  Since our main focus is on solving \eqref{eq:SE3}, we explain preconditioners for \eqref{eq:SE3} in detail
in the following.  We mention \cite{ChinVannelli1994,FreundJarre1997,FreundJarreMizuno1999,BergamaschiGondzioZilli2004,OliveiraSorensen2005,BergamaschiGondzioVenturinZilli,rees2007preconditioner,al2008preconditioning,al2009convergence} as literature related to preconditioners for \eqref{eq:SE2}.

For the iterative solution of \eqref{eq:SE3}, the conjugate gradient (CG) method \cite{HestenesStiefel1952} has been applied with diagonal scaling preconditioners \cite{CarpenterShanno1993,PortugalResendeVeigaJudice2000,Korzak2000} or incomplete Cholesky preconditioners \cite{Mehrotra1992b,KarmarkarRamakrishnan1991,ChinVannelli1994,MehrotraWang1996}. 
LSQR with a preconditioner was used in \cite{GillMurraySaundersTomlinWright1986}. 
A matrix-free method of using CG for least squares (CGLS) preconditioned by a partial Cholesky decomposition was proposed in \cite{Gondzio2012a}.
In \cite{Cui2009th}, a preconditioner based on Greville's method \cite{CuiHayamiYin2011} for generalized minimal residual (GMRES) method was applied.
Suitable preconditioners were also introduced for particular fields such as the minimum-cost network flow problem in \cite{ResendeVeiga1993,JudicePatricioPortugalResendeVeiga2003,MonteiroONeal2003,MonteiroONealTsuchiya2004}.
One may refer to \cite{DApuzzoDeSimoneDiSerafino2010} for a review on the application of numerical linear algebra algorithms to the solutions of KKT systems in the optimization context.

In this paper, we propose to solve \eqref{eq:SE3} using Krylov subspace methods preconditioned by stationary inner-iterations recently proposed for least squares problems in \cite{HayamiYinIto2010,MorikuniHayami2013,MorikuniHayami2015}. 
In Section \ref{sec:lp_InteriorPoint}, we briefly describe the framework of Mehrotra's predictor-corrector interior-point algorithm we implemented and the normal equations arising from this algorithm. 
In Section \ref{sec:lp_InnerIterKrylov}, we specify the application of our method to the normal equations. 
In Section \ref{sec:lp_NumericalExperiments}, we present numerical results comparing our method with a modified sparse Cholesky method, three direct solvers in \textsc{CVX}, a major public package for specifying and solving convex programs \cite{GrantBoyd2014,GrantBoyd2008}, and direct and iterative solvers in PDCO \cite{saunders2002pdco}. The testing problems include the typical LP problems from the \textsc{Netlib}, \textsc{Qaplib} and \textsc{Mittelmann} collections in \cite{DavisHu2011} and basis pursuit problems from the package Atomizer \cite{chen1995atomizer}.
In Section \ref{sec:lp_conclusion}, we conclude the paper. 

Throughout, we use bold lower case letters for column vectors. 
We denote quantities related to the $k$th interior-point iteration by using a superscript with round brackets, e.g.,~$\boldsymbol{x}^{(k)}$, the $k$th iteration of Krylov subspace methods by using a subscript without brackets, e.g.,~$\boldsymbol{x}_{k}$, and the $k$th inner iteration by using a superscript with angle brackets, e.g.,~$\boldsymbol{x}^{\langle k\rangle}$. $\mathcal{R}(A)$ denotes the range space of a matrix $A$. 
$\kappa(A)$ denotes the condition number $\kappa(A) = \sigma_1(A)/\sigma_r(A)$, where $\sigma_1(A)$ and $\sigma_r(A)$ denote the maximum and minimum nonzero singular values of $A$, respectively.
$\mathcal{K}_k (A, \boldsymbol{b}) = \mathrm{span} \lbrace  \boldsymbol{b}, A \boldsymbol{b}, \dots, A^{k-1} \boldsymbol{b} \rbrace$ denotes the Krylov subspace of order $k$.

\section{Interior-point algorithm and the normal equations}\label{sec:lp_InteriorPoint}
We implement an infeasible version of Mehrotra's predictor-corrector method \cite{Mehrotra1992}, which has been established as a standard in this area \cite{LustigMarstenShanno1992,LustigMarstenShanno1994,Wright1997,MehrotraLi2005}.
Note that our method can be applied to other interior-point methods (see, e.g.,~\cite{Wright1997} for more interior-point methods) whose directions are computed via the normal equations \eqref{eq:SE3}.

\subsection{Mehrotra's predictor-corrector algorithm}
In this method, the centering parameter $\sigma$ is determined by dividing each step into two stages. 

In the first stage, we solve for the affine direction $(\Delta\boldsymbol{x}_\mathrm{af}, \Delta\boldsymbol{y}_\mathrm{af}, \Delta\boldsymbol{s}_\mathrm{af})$
\begin{equation}\label{eq:SE3_predictor}
\begin{bmatrix} \boldsymbol{0} & A^\mathsf{T} & I \\ A & \boldsymbol{0} & \boldsymbol{0} \\ S & \boldsymbol{0} & X \end{bmatrix}  \left[ \begin{array}{c} \Delta\boldsymbol{x}_\mathrm{af} \\ \Delta\boldsymbol{y}_\mathrm{af} \\ \Delta\boldsymbol{s}_\mathrm{af} \end{array} \right]= \left[\begin{array}{c} \boldsymbol{r}_\mathrm{d} \\ \boldsymbol{r}_\mathrm{p} \\ -XS\boldsymbol{e} \end{array} \right],
\end{equation}
and measure its progress in reducing $\mu$. 
If the affine direction makes large enough progress without violating the nonnegative boundary \eqref{eq:KKT4}, then $\sigma$ is assigned a small value. 
Otherwise, $\sigma$ is assigned a larger value to steer the iterate to be more centered in the strictly positive region. 

In the second stage, we solve for the corrector direction $(\Delta\boldsymbol{x}_\mathrm{cc}, \Delta\boldsymbol{y}_\mathrm{cc}, \Delta\boldsymbol{s}_\mathrm{cc})$
\begin{equation}\label{eq:SE3_corrector}
\begin{bmatrix}\boldsymbol{0} & A^\mathsf{T} & I \\ A & \boldsymbol{0} & \boldsymbol{0} \\ S & \boldsymbol{0} & X \end{bmatrix}  \left[ \begin{array}{c} \Delta\boldsymbol{x}_\mathrm{cc} \\ \Delta\boldsymbol{y}_\mathrm{cc} \\ \Delta\boldsymbol{s}_\mathrm{cc} \end{array} \right]= \left[\begin{array}{c} \boldsymbol{0} \\ \boldsymbol{0} \\ -\Delta X_\mathrm{af}\Delta S_\mathrm{af}\boldsymbol{e}+\sigma\mu_\mathrm{af} \boldsymbol{e}\end{array} \right],
\end{equation}
where $\Delta X_\mathrm{af} = \mathrm{diag}(\Delta \boldsymbol{x}_\mathrm{af})$, $\Delta S_\mathrm{af} = \mathrm{diag}(\Delta \boldsymbol{s}_\mathrm{af})$ and $\sigma$ is determined according to the solution in the first stage. 
Finally, we update the current iterate along the linear combination of the two directions.

In our implementation of the interior-point method, 
we adopt Mehrotra's predictor-corrector algorithm as follows.

\begin{algorithm}
\caption{Mehrotra's predictor-corrector algorithm.}\label{algorithm:Mehrotra}
\begin{algorithmic}[1] 
\STATE{Given $(\boldsymbol{x}^{(0)}, \boldsymbol{y}^{(0)}, \boldsymbol{s}^{(0)})$ with $(\boldsymbol{x}^{(0)}, \boldsymbol{s}^{(0)})>\boldsymbol{0}$.}
\FOR{$k = 0,1,2, \dots$ until convergence,}
\STATE{$\mu^{(k)} := {\boldsymbol{x}^{(k)}}^\mathsf{T} \boldsymbol{s}^{(k)}/n$ }    {\COMMENT{\texttt{the predictor stage}}}
\STATE{Solve \eqref{eq:SE3_predictor} for the affine direction $(\Delta\boldsymbol{x}_\mathrm{af}, \Delta\boldsymbol{y}_\mathrm{af}, \Delta\boldsymbol{s}_\mathrm{af})$.\label{algorithmline:predictor}}
\STATE{Compute $\alpha_\mathrm{p}, \; \alpha_\mathrm{d}$. \label{algorithmline:stepLength_1}}
\IF{$\min{(\alpha_\mathrm{p},\; \alpha_\mathrm{d})}\geq 1$}
\STATE{$\sigma := 0$, $\left(\Delta\boldsymbol{x}^{(k)}, \Delta\boldsymbol{y}^{(k)}, \Delta\boldsymbol{s}^{(k)}\right) := \left(\Delta\boldsymbol{x}_\mathrm{af}, \Delta\boldsymbol{y}_\mathrm{af}, \Delta\boldsymbol{s}_\mathrm{af}\right)$}
\ELSE 
\STATE{Set $\mu_\mathrm{af}$ and $\sigma := \text{a small value, e.g., } 0.208$.\label{algorithmline:sigma}} {\COMMENT{\texttt{the corrector stage}}}
\STATE{Solve \eqref{eq:SE3_corrector} for the corrector direction $(\Delta\boldsymbol{x}_\mathrm{cc}, \Delta\boldsymbol{y}_\mathrm{cc}, \Delta\boldsymbol{s}_\mathrm{cc})$.\label{algorithmline:corrector}}
\STATE{$\left(\Delta\boldsymbol{x}^{(k)}, \Delta\boldsymbol{y}^{(k)}, \Delta\boldsymbol{s}^{(k)}\right) := (\Delta\boldsymbol{x}_\mathrm{af}, \Delta\boldsymbol{y}_\mathrm{af}, \Delta\boldsymbol{s}_\mathrm{af}) + (\Delta\boldsymbol{x}_\mathrm{cc}, \Delta\boldsymbol{y}_\mathrm{cc}, \Delta\boldsymbol{s}_\mathrm{cc})$}
\ENDIF
\STATE{Compute $\hat{\alpha}_\mathrm{p}, \; \hat{\alpha}_\mathrm{d}$.\label{algorithmline:stepLength_2}}
\STATE{$\boldsymbol{x}^{(k+1)} := \boldsymbol{x}^{(k)} + \hat{\alpha}_\mathrm{p} \Delta\boldsymbol{x}^{(k)}, \left(\boldsymbol{y}^{(k+1)}, \boldsymbol{s}^{(k+1)}\right) := \left(\boldsymbol{y}^{(k)}, \boldsymbol{s}^{(k)}\right) + \hat{\alpha}_\mathrm{d}\left(\Delta\boldsymbol{y}^{(k)}, \Delta\boldsymbol{s}^{(k)}\right)$}
\ENDFOR
\end{algorithmic}
\end{algorithm}

In line \ref{algorithmline:stepLength_1} in Algorithm \ref{algorithm:Mehrotra}, the step lengths $\alpha_\mathrm{p}, \; \alpha_\mathrm{d}$ are computed by
\begin{equation}\label{eq:steplength}
\alpha_\mathrm{p} = \min{ (1,\eta \min_{i: \Delta x_i<0}(-\frac{ x_i}{\Delta x_i}))},\  \
\alpha_\mathrm{d} = \min{ (1, \eta \min_{i: \Delta s_i<0}(-\frac{ s_i}{\Delta s_i}))},
\end{equation}
where $(\Delta\boldsymbol{x}, \Delta\boldsymbol{s}) = (\Delta\boldsymbol{x}_\mathrm{af}, \Delta\boldsymbol{s}_\mathrm{af}), \eta \in [0.9, 1)$.

In line \ref{algorithmline:sigma}, the quantity $\mu_\mathrm{af}$ is computed by
\begin{equation*}
\mu_\mathrm{af} = (\boldsymbol{x}^{(k)} + \alpha_\mathrm{p}\Delta\boldsymbol{x}_\mathrm{af})^\mathsf{T} (\boldsymbol{s}^{(k)} + \alpha_\mathrm{d}\Delta\boldsymbol{s}_\mathrm{af})/n.
\end{equation*}
In the same line, the parameter $\sigma$ is chosen as $\sigma = \min{ (0.208, (\mu_\mathrm{af}/\mu^{(k)})^2)}$ 
in the early phase of the interior-point iterations. The value $0.208$ and the range $[0.9, 1)$ for $\eta$ are adopted by the LIPSOL package \cite{Zhang1998}.
In the late phase of the interior-point iterations, $\sigma$ is chosen as approximately 10 times the error measure $\Gamma^{(k)}$ which is defined as:
\begin{equation}\label{eq:errorMeasureDef}
\quad\Gamma^{(k)} := \max \left\{\mu^{(k)}, \frac{\|\boldsymbol{b}-A\boldsymbol{x}^{(k)}\|_2}{\max\left\{\|\boldsymbol{b}\|_2,1\right\}}, \frac{\|\boldsymbol{c}-\boldsymbol{s}^{(k)}-A^\mathsf{T}\boldsymbol{y}^{(k)}\|_2}{\max\left\{\|\boldsymbol{c}\|_2,1\right\}}\right\}.
\end{equation}
Here the distinction between \textit{early} and \textit{late} phases is when $\Gamma^{(k)}$ is more or less than $10^{-3}$.

In line \ref{algorithmline:stepLength_2}, we first compute trial step lengths $\alpha_\mathrm{p}, \alpha_\mathrm{d}$ using equations \eqref{eq:steplength} with $(\Delta\boldsymbol{x}, \Delta\boldsymbol{s}) = (\Delta\boldsymbol{x}^{(k)}, \Delta\boldsymbol{s}^{(k)})$. 
Then, we gradually reduce $\alpha_\mathrm{p}, \alpha_\mathrm{d}$ to find the largest step lengths that can ensure the centrality of the updated iterates, i.e., to find the maximum $\hat{\alpha}_\mathrm{p}, \; \hat{\alpha}_\mathrm{d}$ that satisfy
\begin{equation*}
\min_{i}( x_i + \hat{\alpha}_\mathrm{p} \Delta x_i)( s_i + \hat{\alpha}_\mathrm{d} \Delta s_i) \geq \phi(\boldsymbol{x} + \hat{\alpha}_\mathrm{p} \Delta\boldsymbol{x})^\mathsf{T}(\boldsymbol{s} + \hat{\alpha}_\mathrm{d} \Delta\boldsymbol{s})/n,
\end{equation*}
where $\phi$ is typically chosen as $10^{-5}$.

\subsection{The normal equations in the interior-point algorithm}
We consider modifying Algorithm \ref{algorithm:Mehrotra} so that it is not necessary to update $\boldsymbol{y}^{(k)}$.
Since we assume the existence of an optimal solution to problem \eqref{eq:PrimalDualPair}, we have $\boldsymbol{b} \in \mathcal{R}(A)$.
Let $D:=S^{-1/2} X^{1/2}$ and $\mathcal{A} := AD$. 
Problem \eqref{eq:SE3} with $\Delta \boldsymbol{w} = \mathcal{A}^\mathsf{T}\Delta \boldsymbol{y}$ (the normal equations of the second kind) is equivalent to
\begin{align}
  \min \| \Delta \boldsymbol{w} \|_2 \quad \mbox{subject to} \quad \mathcal{A} \Delta \boldsymbol{w} = \boldsymbol{f},
  \label{eq:minnrmprob}
\end{align}
where $\boldsymbol{f}:=\boldsymbol{r}_\mathrm{p} - AS^{-1}(\boldsymbol{r}_\mathrm{c} - X\boldsymbol{r}_\mathrm{d})$.

In the predictor stage, the problem \eqref{eq:SE3_predictor} is equivalent to first solving \eqref{eq:minnrmprob} for $\Delta\boldsymbol{w}_\mathrm{af}$ with $\Delta\boldsymbol{w} = \Delta\boldsymbol{w}_\mathrm{af},\;\boldsymbol{f} = \boldsymbol{f}_\mathrm{af} := \boldsymbol{b}+AS^{-1}X\boldsymbol{r}_\mathrm{d}$, and then updating the remaining unknowns by 
\begin{subequations} \label{eq:step_predictor}
\begin{align}
\label{eq:aff2}\Delta \boldsymbol{s}_\mathrm{af} &= \boldsymbol{r}_\mathrm{d} - D^{-1} \Delta \boldsymbol{w}_\mathrm{af},\\
\label{eq:aff3}\Delta \boldsymbol{x}_\mathrm{af} &= -D^2\Delta\boldsymbol{s}_\mathrm{af} - \boldsymbol{x}.
\end{align}
\end{subequations}

In the corrector stage, the problem \eqref{eq:SE3_corrector} is equivalent to first solving \eqref{eq:minnrmprob} for $\Delta\boldsymbol{w}_\mathrm{cc}$ with $\Delta\boldsymbol{w} = \Delta\boldsymbol{w}_\mathrm{cc},\;\boldsymbol{f} = \boldsymbol{f}_\mathrm{cc} := AS^{-1}\Delta X_\mathrm{af}\Delta S_\mathrm{af}\boldsymbol{e} - \sigma\mu AS^{-1}\boldsymbol{e}$, and then updating the remaining unknowns by 
\begin{subequations}  \label{eq:step_corrector}
\begin{align}
\label{eq:cc2}\Delta\boldsymbol{s}_\mathrm{cc} &= - D^{-1} \Delta\boldsymbol{w}_\mathrm{cc},\\
\label{eq:cc3}\Delta\boldsymbol{x}_\mathrm{cc} &= -D^2\Delta\boldsymbol{s}_\mathrm{cc} -S^{-1}\Delta X_\mathrm{af}\Delta \boldsymbol{s}_\mathrm{af} + \sigma\mu S^{-1}\boldsymbol{e}.
\end{align}
\end{subequations}
By solving \eqref{eq:minnrmprob} for $\Delta\boldsymbol{w}$ instead of solving \eqref{eq:SE3} for $\Delta\boldsymbol{y}$, we can compute $\Delta\boldsymbol{s}_\mathrm{af}$, $\Delta\boldsymbol{ x}_\mathrm{af}$, $\Delta\boldsymbol{ s}_\mathrm{cc}$, and $\Delta\boldsymbol{x}_\mathrm{cc}$ and can save 1MV in \eqref{eq:aff2} and another in \eqref{eq:cc2} if a predictor step is performed per interior-point iteration. Here, MV denotes the computational cost required for one matrix-vector multiplication. 
\begin{remark}
For solving an interior-point step from the condensed step equation~\eqref{eq:SE3} using a suited Krylov subspace method,
updating $(\boldsymbol{x},\boldsymbol{w},\boldsymbol{s})$ rather than $(\boldsymbol{x},\boldsymbol{y},\boldsymbol{s})$ can save 1MV each interior-point iteration.
\end{remark}
Note that in the predictor and corrector stages, problem \eqref{eq:minnrmprob} has the same matrix but different right-hand sides.
We introduce methods for solving it in the next section.

\section{Application of inner-iteration preconditioned Krylov subspace methods}\label{sec:lp_InnerIterKrylov}

In lines \ref{algorithmline:predictor} and \ref{algorithmline:corrector} of Algorithm \ref{algorithm:Mehrotra}, the linear system \eqref{eq:minnrmprob} needs to be solved, with its matrix becoming increasingly ill-conditioned as the interior-point iterations proceed.
In this section, we focus on applying inner-iteration preconditioned Krylov subspace methods to \eqref{eq:minnrmprob} because they are advantageous in dealing with ill-conditioned sparse matrices.
The methods to be discussed are the preconditioned CG and MINRES methods \cite{HestenesStiefel1952,PaigeSaunders1975} applied to the normal equations of the second kind ((P)CGNE and (P)MRNE, respectively) \cite{Craig1955,MorikuniHayami2015},
and the right-preconditioned generalized minimal residual method (AB-GMRES) \cite{HayamiYinIto2010,MorikuniHayami2015}.

Consider solving linear system $\mathbf{A} \mathbf{x} = \mathbf{b}$, where $\mathbf{A} \in \mathbf{R}^{n \times n}$. First, the conjugate gradient (CG) method \cite{HestenesStiefel1952} is an iterative method for such problems when $\mathbf{A}$ is a symmetric and positive (semi)definite matrix and $\mathbf{b} \in \mathcal{R}(\mathbf{A})$. CG starts with an initial approximate solution $\mathbf{x}_0 \in \mathbb{R}^n$ and determines the $k$th iterate $\mathbf{x}_k \in \mathbb{R}^n$ by minimizing $\| \mathbf{x}_k - \mathbf{x}_* \|^2_\mathbf{A}$ over the space $\mathbf{x}_0 + \mathcal{K}_k (\mathbf{A}, \mathbf{r}_0)$, where $\mathbf{r}_0 = \mathbf{b} - \mathbf{A} \mathbf{x}_0$,
$\mathbf{x}_*$ is a solution of $\mathbf{A} \mathbf{x} = \mathbf{b}$, and $\| \mathbf{x}_k - \mathbf{x}_* \|^2_\mathbf{A} := (\mathbf{x}_k - \mathbf{x}_*)^\mathsf{T}\mathbf{A}(\mathbf{x}_k - \mathbf{x}_*)$. 

MINRES \cite{PaigeSaunders1975} is another iterative method for solving such problems but only requires  $\mathbf{A}$ to be symmetric. MINRES with $\mathbf{x}_0$ determines the $k$th iterate $\mathbf{x}_k$ by minimizing $\|  \mathbf{b} - \mathbf{A} \mathbf{x} \|_2$ over the same space as CG.

Third, the generalized minimal residual (GMRES) method \cite{SaadSchultz1986} only requires $\mathbf{A}$ to be square. GMRES with $\mathbf{x}_0$ determines the $k$th iterate $\mathbf{x}_k$ by minimizing $\| \mathbf{b} - \mathbf{A} \mathbf{x} \|_2$ over $\boldsymbol{x}_0 + \mathcal{K}_k (\mathbf{A}, \mathbf{r}_0)$.

\subsection{Application of inner-iteration preconditioned CGNE and MRNE methods}\label{sec:CGNEMRNE}
We first introduce CGNE and MRNE.
Let $\mathbf{A} = \mathcal{A} \mathcal{A}^\mathsf{T}$, $\mathbf{x} = \Delta \boldsymbol{y}_\mathrm{af}$, $\mathbf{b} = \boldsymbol{f}_\mathrm{af}$, and $\Delta \boldsymbol{w}_\mathrm{af} = \mathcal{A}^\mathsf{T} \Delta \boldsymbol{y}_\mathrm{af}$ for the predictor stage, 
and similarly, let $\mathbf{A} = \mathcal{A} \mathcal{A}^\mathsf{T}$, $\mathbf{x} = \Delta \boldsymbol{y}_\mathrm{cc}$, $\mathbf{b} = \boldsymbol{f}_\mathrm{cc}$, and $\Delta \boldsymbol{w}_\mathrm{cc} = \mathcal{A}^\mathsf{T} \Delta \boldsymbol{y}_\mathrm{cc}$ for the corrector stage.
CG and MINRES applied to systems $\mathbf{A} \mathbf{x} = \mathbf{b}$ are CGNE and MRNE, respectively. 
With these settings, let the initial solution $\Delta \boldsymbol{w}_0 \in \mathcal{R}(\mathcal{A}^\mathsf{T})$ in both stages, and denote the initial residual by $\boldsymbol{g}_0:=\boldsymbol{f}-\mathcal{A}\Delta\boldsymbol{w}_0$.
CGNE and MRNE can solve \eqref{eq:minnrmprob} without forming $\mathcal{A} \mathcal{A}^\mathsf{T}$ explicitly. 

Concretely, CGNE gives the $k$th iterate $\Delta \boldsymbol{w}_k$ such that $\| \Delta \boldsymbol{w}_k - \Delta \boldsymbol{w}_* \|_2 = \linebreak \min_{\Delta \boldsymbol{w} \in \Delta \boldsymbol{w}_0 + \mathcal{K}_k (\mathcal{A}^\mathsf{T} \! \mathcal{A}, \mathcal{A}^\mathsf{T} \boldsymbol{g}_0)} \| \Delta \boldsymbol{w} - \Delta \boldsymbol{w}_* \|_2$, where $\Delta \boldsymbol{w}_*$ is the minimum-norm solution of $\mathcal{A} \Delta \boldsymbol{w} = \boldsymbol{f}$ for $\Delta \boldsymbol{w}_0 \in \mathcal{R}(\mathcal{A}^\mathsf{T})$ and $\boldsymbol{f} \in \mathcal{R}(\mathcal{A})$.
MRNE gives the $k$th iterate $\Delta \boldsymbol{w}_k$ such that $\| \boldsymbol{f} - \mathcal{A} \Delta \boldsymbol{w}_k \|_2 = \min_{\Delta \boldsymbol{w} \in \Delta \boldsymbol{w}_0 + \mathcal{K}_k (\mathcal{A}^\mathsf{T} \! \mathcal{A}, \mathcal{A}^\mathsf{T} \! \boldsymbol{g}_0)}  \| \boldsymbol{f} - \mathcal{A} \Delta \boldsymbol{w} \|_2$. 

We use inner-iteration preconditioning for CGNE and MRNE methods.
The following is a brief summary of the part of \cite{MorikuniHayami2015} where the inner-outer iteration method is analyzed. 
We give the expressions for the inner-iteration preconditioning and preconditioned matrices to state the conditions under which the former is SPD.
Let $M$ be a symmetric nonsingular splitting matrix of $\mathcal{A} \mathcal{A}^\mathsf{T}$ such that $\mathcal{A} \mathcal{A}^\mathsf{T}  = M - N$.
Denote the inner-iteration matrix by $H = M^{-1} N$.
The inner-iteration preconditioning and preconditioned matrices are $C^{\langle\ell\rangle} = \sum_{i = 0}^{\ell - 1} H^i M^{-1}$ and $\mathcal{A} \mathcal{A}^\mathsf{T} C^{\langle\ell\rangle} = M\sum_{i=0}^{\ell-1} (I- H) H^iM^{-1} = M(I - H^\ell)M^{-1}$, respectively.
If $C^{\langle\ell\rangle}$ is nonsingular, then $\mathcal{A} \mathcal{A}^\mathsf{T} C^{\langle\ell\rangle} \boldsymbol{u} = \boldsymbol{f}$, $\boldsymbol{z} = C^{\langle\ell\rangle} \boldsymbol{u}$ is equivalent to $\mathcal{A} \mathcal{A}^\mathsf{T} \boldsymbol{z} = \boldsymbol{f}$ for all $\boldsymbol{f} \in \mathcal{R}(\mathcal{A})$.
For $\ell$ odd, $C^{\langle\ell\rangle}$ is symmetric and positive definite (SPD) if and only if the inner-iteration $M$ is SPD;
for $\ell$ even, $C^{\langle\ell\rangle}$ is SPD if and only if $M + N$ is SPD \cite[Theorem 2.8]{Morikuni2015a,Morikuni2018}.
We give Algorithms \ref{algorithm:CGNE}, \ref{algorithm:MRNE} for CGNE and MRNE preconditioned by inner iterations \cite[Algorithms E.3, E.4]{MorikuniHayami2015}.

\begin{algorithm}
\caption{CGNE method preconditioned by inner iterations.} \label{algorithm:CGNE}
\begin{algorithmic}[1]
\STATE{Let $\Delta\boldsymbol{w}_0$ be the initial approximate solution, and $\boldsymbol{g}_0:=\boldsymbol{f}-\mathcal{A}\Delta\boldsymbol{w}_0$.}
\STATE{Apply $\ell$ steps of a stationary iterative method to $\mathcal{A}\mathcal{A}^\mathsf{T}\boldsymbol{z}=\boldsymbol{g}_0,\;\boldsymbol{u}=\mathcal{A}^\mathsf{T}\boldsymbol{z}$ to obtain $\boldsymbol{z}_0:=\mathcal{C}^{\langle\ell\rangle}\boldsymbol{g}_0$ and $\boldsymbol{u}_0:=\mathcal{A}^\mathsf{T}\boldsymbol{z}_0$.\label{algorithmline:CGNEprecond1}}
\STATE{$\boldsymbol{q}_0:=\boldsymbol{u}_0, \gamma_0:=(\boldsymbol{g}_0, \boldsymbol{z}_0)$}
\FOR{$k = 0, 1, 2,\dots $ until convergence,}
\STATE{$\alpha_k:=\gamma_k/(\boldsymbol{q}_k, \boldsymbol{q}_k),\;\; \Delta\boldsymbol{w}_{k+1}:=  \Delta\boldsymbol{w}_{k} + \alpha \boldsymbol{q}_k,\;\; \boldsymbol{g}_{k+1}:=\boldsymbol{g}_k-\alpha_k\mathcal{A}\boldsymbol{q}_k$}
\STATE{Apply $\ell$ steps of a stationary iterative method to $\mathcal{A}\mathcal{A}^\mathsf{T}\boldsymbol{z}=\boldsymbol{g}_{k+1}$ to obtain $\boldsymbol{z}_{k+1}:=\mathcal{C}^{\langle\ell\rangle}\boldsymbol{g}_{k+1}$ and $\boldsymbol{u}_{k+1}:=\mathcal{A}^\mathsf{T}\boldsymbol{z}_{k+1}$.\label{algorithmline:CGNEprecond2}}
\STATE{$\gamma_{k+1}:=(\boldsymbol{g}_{k+1}, \boldsymbol{z}_{k+1}), \;\; \beta_k := \gamma_{k+1}/\gamma_{k},\;\; \boldsymbol{q}_{k+1}:=\boldsymbol{u}_{k+1}+ \beta_k\boldsymbol{q}_k$}
\ENDFOR
\end{algorithmic}
\end{algorithm}
\begin{algorithm}
\caption{MRNE method preconditioned by inner iterations.}\label{algorithm:MRNE}
\begin{algorithmic}[1]
\STATE{Let $\Delta\boldsymbol{w}_0$ be the initial approximate solution, and $\boldsymbol{g}_0:=\boldsymbol{f}-\mathcal{A}\Delta\boldsymbol{w}_0$.}
\STATE{Apply $\ell$ steps of a stationary iterative method to $\mathcal{A}\mathcal{A}^\mathsf{T}\boldsymbol{u}=\boldsymbol{g}_0, \;\boldsymbol{q}=\mathcal{A}^\mathsf{T}\boldsymbol{u}$ to obtain $\boldsymbol{q}_0:=\mathcal{A}^\mathsf{T}\mathcal{C}^{\langle\ell\rangle}\boldsymbol{g}_0$.\label{algorithmline:MRNEprecond1}}
\STATE{$\boldsymbol{p}_0:= \boldsymbol{q}_0, \; \gamma_0:=\|\boldsymbol{q}_0\|_2^2$}
\FOR{$k = 1,2,\dots$ until convergence,}
\STATE{$\boldsymbol{t}_k:=\mathcal{A}\boldsymbol{p}_k$}
\STATE{Apply $\ell$ steps of a stationary iterative method to $\mathcal{A}\mathcal{A}^\mathsf{T}\boldsymbol{u}=\boldsymbol{t}_k, \;\boldsymbol{v} = \mathcal{A}^\mathsf{T}\boldsymbol{u}$ to obtain $\boldsymbol{v}_k:=\mathcal{A}^\mathsf{T}\mathcal{C}^{\langle\ell\rangle}\boldsymbol{t}_k$.\label{algorithmline:MRNEprecond2}}
\STATE{$\alpha_{k}:=\gamma_k/(\boldsymbol{v}_k, \boldsymbol{p}_k), \;\Delta\boldsymbol{w}_k := \Delta\boldsymbol{w}_k + \alpha_k\boldsymbol{p}_k, \;\boldsymbol{g}_{k+1} := \boldsymbol{g}_k - \alpha_k\boldsymbol{t}_k, \boldsymbol{q}_{k+1}:=\boldsymbol{q}_k-\alpha_k\boldsymbol{v}_k$}
\STATE{$\gamma_k:= \|\boldsymbol{q}_{k+1}\|_2^2, \;\beta_{k}:= \gamma_{k+1}/\gamma_{k}, \;\boldsymbol{p}_{k+1}:=\boldsymbol{q}_k + \beta_k\boldsymbol{p}_k$}
\ENDFOR
\end{algorithmic}
\end{algorithm}

\subsection
{Application of inner-iteration preconditioned AB-GMRES meth\-od}\label{sec:ABGMRES}
Next, we introduce AB-GMRES.
GMRES can solve a square linear system transformed from the rectangular system $\mathcal{A} \Delta \boldsymbol{w}_\mathrm{af} = \boldsymbol{f}_\mathrm{af}$ in the predictor stage and $\mathcal{A} \Delta \boldsymbol{w}_\mathrm{cc} = \boldsymbol{f}_\mathrm{cc}$ in the corrector stage by using a rectangular right-preconditioning matrix that does not necessarily have to be $\mathcal{A^\mathsf{T}}$.
Let $\mathcal{B} \in \mathbb{R}^{n \times m}$ be a preconditioning matrix for $\mathcal{A}$.
Then, AB-GMRES corresponds to GMRES \cite{SaadSchultz1986} applied to 
\begin{align*}
  \mathcal{A} \mathcal{B} \boldsymbol{z} = \boldsymbol{f}, \quad \Delta\boldsymbol{w} = \mathcal{B} \boldsymbol{z},
\end{align*}
which is equivalent to the minimum-norm solution to the problem \eqref{eq:minnrmprob},
for all $\boldsymbol{f} \in \mathcal{R}(\mathcal{A})$ if $\mathcal{R}(\mathcal{B}) = \mathcal{R}(\mathcal{A}^\mathsf{T})$ \cite[Theorem 5.2]{MorikuniHayami2015}, where $\Delta \boldsymbol{w} = \Delta \boldsymbol{w}_\mathrm{af}$ or $\Delta \boldsymbol{w}_\mathrm{cc}$, $\boldsymbol{f} = \boldsymbol{f}_\mathrm{af}$ or $\boldsymbol{f}_\mathrm{cc}$, respectively.
AB-GMRES gives the $k$th iterate $\Delta \boldsymbol{w}_k = \mathcal{B} \boldsymbol{z}_k$ such that $\boldsymbol{z}_k=\arg\!\min_{\boldsymbol{z} \in \boldsymbol{z}_0 + \mathcal{K}_k (\mathcal{A} \mathcal{B}, \boldsymbol{g}_0)} \| \boldsymbol{f} - \mathcal{A} \mathcal{B} \boldsymbol{z} \|_2$, where $\boldsymbol{z}_0$ is the initial iterate and $\boldsymbol{g}_0 = \boldsymbol{f} - \mathcal{A} \mathcal{B} \boldsymbol{z}_0$.

Specifically, we apply AB-GMRES preconditioned by inner iterations \cite{MorikuniHayami2013,MorikuniHayami2015} to \eqref{eq:minnrmprob}.
This method was shown to outperform previous methods on ill-conditioned and rank-deficient problems.
We give expressions for the inner-iteration preconditioning and preconditioned matrices.
Let $M$ be a nonsingular splitting matrix such that $\mathcal{A} \mathcal{A}^\mathsf{T}  = M - N$.
Denote the inner-iteration matrix by $H = M^{-1} N$.
With $C^{\langle\ell\rangle} = \sum_{i = 0}^{\ell - 1} H^i M^{-1}$, the inner-iteration preconditioning and preconditioned matrices are $\mathcal{B}^{\langle\ell\rangle} = \mathcal{A}^\mathsf{T} C^{\langle\ell\rangle}$ and $\mathcal{A} \mathcal{B}^{\langle\ell\rangle} = \sum_{i=0}^{\ell-1} (I- H) H^i = M(I - H^\ell)M^{-1}$, respectively.
If the inner-iteration matrix $H$ is semiconvergent, i.e., $\lim_{i \rightarrow \infty} H^i$ exists, then AB-GMRES preconditioned by the inner-iterations determines the minimum-norm solution of $\mathcal{A} \Delta \boldsymbol{w} = \boldsymbol{f}$ without breakdown for all $\boldsymbol{f} \in \mathcal{R}(\mathcal{A})$ and for all $\Delta \boldsymbol{w}_0 \in \mathcal{R}(\mathcal{A}^\mathsf{T})$ \cite[Theorem 5.5]{MorikuniHayami2015}.
The inner-iteration preconditioning matrix $\mathcal{B}^{\langle\ell\rangle}$ works on $\mathcal{A}$ in AB-GMRES as in Algorithm \ref{algorithm:AB-GMREStoNE} \cite[Algorithm 5.1]{MorikuniHayami2015}.

\begin{algorithm}[H]
\caption{AB-GMRES method preconditioned by inner iterations.} \label{algorithm:AB-GMREStoNE}
\begin{algorithmic}[1]
\STATE{Let $\Delta\boldsymbol{w}_0 \in \mathbb{R}^n$ be the initial approximate solution, and $\boldsymbol{g}_0:=\boldsymbol{f}-\mathcal{A}\Delta\boldsymbol{w}_0$.}
\STATE{$\beta:= \|\boldsymbol{g}_0\|_2, \;\; \boldsymbol{v}_1:=\boldsymbol{r}_0/\beta$}
\FOR{$k = 1,2,\dots$ until convergence,}
\STATE{Apply $\ell$ steps of a stationary iterative method to $\mathcal{A}\mathcal{A}^\mathsf{T}\boldsymbol{p}=\boldsymbol{v}_k,\;\boldsymbol{z}=\mathcal{A}^\mathsf{T}\boldsymbol{p}$ to obtain $\boldsymbol{z}_k:=\mathcal{B}^{\langle\ell\rangle}\boldsymbol{v}_k$.\label{algorithmline:ABprecond1}}
\STATE{$\boldsymbol{u}_k:=\mathcal{A}\boldsymbol{z}_k$}
\FOR{$i = 1,2,\dots, k$,}
\STATE{$h_{i,k}:=(\boldsymbol{u}_k,\boldsymbol{v}_i),\;\boldsymbol{u}_k:=\boldsymbol{u}_k-h_{i,k}\boldsymbol{v}_i$}
\ENDFOR
\STATE{$h_{k+1,k}:=\|\boldsymbol{u}_k\|_2,\;\boldsymbol{v}_{k+1}:=\boldsymbol{u}_k/h_{k+1,k}$}
\ENDFOR
\STATE{$\boldsymbol{p}_k:=\mathrm{arg} \min_{\boldsymbol{p}\in R^k}{\|\beta \boldsymbol{e_1} - \bar{H}_k\boldsymbol{p}\|_2},\; \boldsymbol{q}_k = [\boldsymbol{v}_1, \boldsymbol{v}_2, \dots, \boldsymbol{v}_k]\boldsymbol{p}_k$}
\STATE{Apply $\ell$ steps of a stationary iterative method to $\mathcal{A}\mathcal{A}^\mathsf{T}\boldsymbol{p}=\boldsymbol{q}_k,\;\boldsymbol{z}=\mathcal{A}^\mathsf{T}\boldsymbol{p}$ to obtain $\boldsymbol{z}':=\mathcal{B}^{\langle\ell\rangle}\boldsymbol{q}_k$. \label{algorithmline:ABprecond2}}
\STATE{$\Delta\boldsymbol{w}_k:=\Delta\boldsymbol{w}_0+\boldsymbol{z}'$}
\end{algorithmic}
\end{algorithm}

Here, $\boldsymbol{v}_1, \boldsymbol{v}_2, \dots, \boldsymbol{v}_k$ are orthonormal, $\boldsymbol{e}_1$ is the first column of the identity matrix, and $\bar{H}_k = \lbrace h_{i, j} \rbrace \in \mathbb{R}^{(k+1) \times k}$.

Note that the left-preconditioned generalized minimal residual method (BA-GMRES) \cite{HayamiYinIto2010,MorikuniHayami2013,MorikuniHayami2015} can be applied to solve the corrector stage problem, which can be written as the normal equations of the first kind
\begin{equation*}
\mathcal{A}\mathcal{A}^\mathsf{T}\Delta\boldsymbol{y}_\mathrm{cc} = \mathcal{A}(SX)^{-1/2}\left( \Delta X_\mathrm{af}\Delta S_\mathrm{af}\boldsymbol{e} - \sigma\mu \boldsymbol{e}\right),
\end{equation*}
or equivalently
\begin{equation}\label{eq:ba_ne}
\min_{\Delta\boldsymbol{y}_\mathrm{cc}} \| \mathcal{A}^\mathsf{T} \Delta\boldsymbol{y}_\mathrm{cc} -(SX)^{-1/2}\left( \Delta X_\mathrm{af}\Delta S_\mathrm{af}\boldsymbol{e} - \sigma\mu \boldsymbol{e}\right) \|_2.
\end{equation}
In fact, this formulation was adopted in \cite{Gondzio2012} and solved by the CGLS method preconditioned by partial Cholesky decomposition that
works in $m$-dimen-sional space.
The BA-GMRES also works in $m$-dimensional space.

The advantage of the inner-iteration preconditioning methods is that we can avoid explicitly computing and storing the preconditioning matrices for $\mathcal{A}$ in \eqref{eq:minnrmprob}.
We present efficient algorithms for specific inner iterations in the next section.

\subsection{SSOR inner iterations for preconditioning the CGNE and MRNE me\-thods}
The inner-iteration preconditioned CGNE and MRNE methods require a symmetric preconditioning matrix.
This is achieved by the SSOR inner-iteration preconditioning, which works on the normal equations of the second kind $\mathcal{A} \mathcal{A}^\mathsf{T} \boldsymbol{z} = \boldsymbol{g}$, $\boldsymbol{u} = \mathcal{A}^\mathsf{T} \boldsymbol{z}$, and its preconditioning matrix $C^{\langle\ell\rangle}$ is SPD for $\ell$ odd for $\omega \in (0, 2)$ \cite[Theorem 2.8]{Morikuni2015a,Morikuni2018}.
This method exploits a symmetric splitting matrix by the forward updates, $i = 1, 2, \dots, m$ in lines \ref{algorithmlinSSOR_l3}--\ref{algorithmlinSSOR_l6} in Algorithm \ref{algorithm:NESOR} and the reverse updates, $i = m, m-1, \dots, 1$, and can be efficiently implemented as the NE-SSOR method \cite{Saad2003}, \cite[Algorithm D.8]{MorikuniHayami2015}.
See \cite{BjorckElfving1979} where SSOR preconditioning for CGNE with $\ell = 1$ is proposed.
Let $\boldsymbol{\alpha}_i^\mathsf{T}$ be the $i$th row vector of $\mathcal{A}$. \
Algorithm \ref{algorithm:NESSOR} shows the NE-SSOR method.
\begin{algorithm}
\caption{NE-SSOR method.}\label{algorithm:NESSOR}
\begin{algorithmic}[1]
\STATE{Let $\boldsymbol{z}^{\langle 0\rangle}=\boldsymbol{0}$ and $\boldsymbol{u}^{\langle 0\rangle}=\boldsymbol{0}$.}
\FOR{$k = 1,2,\dots, \ell$,}
\FOR{$i = 1,2,\dots, m$,}\label{algorithmlinSSOR_l3}
\STATE{$d_i^{\langle k-\frac{1}{2}\rangle}:= \omega[g_i - (\boldsymbol{\alpha}_i,\boldsymbol{u}^{\langle k-1\rangle})]/\|\boldsymbol{\alpha}_i\|_2^2$}
\STATE{$z_i^{\langle k-\frac{1}{2}\rangle}:= z_i^{\langle k-1\rangle} + d_i^{\langle k-\frac{1}{2}\rangle},\boldsymbol{u}^{\langle k-1\rangle}:= \boldsymbol{u}^{\langle k-1\rangle} +d_i^{\langle k-\frac{1}{2}\rangle}\boldsymbol{\alpha}_i$}
\ENDFOR\label{algorithmlinSSOR_l6}

\FOR{$i = m,m-1,\dots, 1$,}
\STATE{$d_i^{\langle k\rangle}:= \omega[g_i - (\boldsymbol{\alpha}_i,\boldsymbol{u}^{\langle k-1\rangle})]/\|\boldsymbol{\alpha}_i\|_2^2$}
\STATE{$z_i^{\langle k\rangle}:= z_i^{\langle k-\frac{1}{2}\rangle} + d_i^{\langle k\rangle},\boldsymbol{u}^{\langle k-1\rangle}:= \boldsymbol{u}^{\langle k-1\rangle} +d_i^{\langle k\rangle}\boldsymbol{\alpha}_i$}
\ENDFOR
\STATE{$\boldsymbol{u}^{\langle k\rangle}:=\boldsymbol{u_{_{}}}^{\langle k-1\rangle}$}
\ENDFOR
\end{algorithmic}
\end{algorithm}

When Algorithm \ref{algorithm:NESSOR} is applied to lines \ref{algorithmline:CGNEprecond1} and \ref{algorithmline:CGNEprecond2} of Algorithm \ref{algorithm:CGNE} and lines \ref{algorithmline:MRNEprecond1} and \ref{algorithmline:MRNEprecond2} of Algorithm \ref{algorithm:MRNE}, the normal equations of the second kind are solved approximately.

\subsection{SOR inner iterations for preconditioning the AB-GMRES method}
Next, we introduce the SOR method applied to the normal equations of the second kind $\mathcal{A} \mathcal{A}^\mathsf{T} \boldsymbol{p} = \boldsymbol{g}$, $\boldsymbol{z} = \mathcal{A}^\mathsf{T} \boldsymbol{p}$ with $\boldsymbol{g} = \boldsymbol{v}_k$ or $\boldsymbol{q}_k$ as used in Algorithm \ref{algorithm:AB-GMREStoNE}. 
If the relaxation parameter $\omega$ satisfies $\omega \in (0, 2)$, then the iteration matrix $H$ of this method is semiconvergent, i.e., $\lim_{i \rightarrow \infty} H^i$ exists \cite{Dax1990}.
An efficient algorithm for this method is called NE-SOR and is given as follows  \cite{Saad2003}, \cite[Algorithm D.7]{MorikuniHayami2015}.

\begin{algorithm}
\caption{NE-SOR method.}\label{algorithm:NESOR}
\begin{algorithmic}[1]
\STATE{Let $\boldsymbol{z}^{\langle 0\rangle}=\boldsymbol{0}$.}
\FOR{$k = 1,2,\dots, \ell$,}
\FOR{$i = 1,2, \dots, m$,}
\STATE{$d_i^{\langle k\rangle}:=\omega[ g_i - (\boldsymbol{\alpha}_i, \;\boldsymbol{z}^{\langle k-1\rangle})] / \|\boldsymbol{\alpha}_i\|_2^2, \;\boldsymbol{z}^{\langle k-1\rangle}:=\boldsymbol{z}^{\langle k-1\rangle}+d_i^{\langle k\rangle}\boldsymbol{\alpha}_i$\label{algorithmline:NESOR_mod}}
\ENDFOR
\STATE{$\boldsymbol{z}^{\langle k\rangle}:=\boldsymbol{z}^{\langle k-1\rangle}$}
\ENDFOR
\end{algorithmic}
\end{algorithm}

When Algorithm \ref{algorithm:NESOR} is applied to lines \ref{algorithmline:ABprecond1} and \ref{algorithmline:ABprecond2} of Algorithm \ref{algorithm:AB-GMREStoNE}, the normal equations of the second kind are solved approximately.

Since the rows of $A$ are required in the NE-(S)SOR iterations, it would be more efficient if $A$ is stored row-wise.

\subsection{\texorpdfstring{Row-scaling of $\mathcal{A}$}{Row-scaling of A}}\label{sec:rowscaling}
Let $\mathcal{D}$ be a diagonal matrix whose diagonal elements are positive. Then, problem \eqref{eq:minnrmprob} is equivalent to
\begin{equation}\label{eq:scaleD}
\min\|\Delta\boldsymbol{w}\|_2 \quad \mbox{subject to} \quad  \mathcal{D}^{-1}\mathcal{A}\Delta\boldsymbol{w} = \mathcal{D}^{-1}\boldsymbol{f}.
\end{equation}
Denote $\hat{\mathcal{A}}:=\mathcal{D}^{-1}\mathcal{A}$ and $\hat{f}:=\mathcal{D}^{-1}\boldsymbol{f}$. Then, the scaled problem \eqref{eq:scaleD} is
\begin{equation}\label{eq:scaled_ls}
\min\|\Delta\boldsymbol{w}\|_2 \quad \mbox {subject to} \quad \hat{\mathcal{A}}\Delta\boldsymbol{w} = \hat{\boldsymbol{f}}.
\end{equation}
If $\hat{\mathcal{B}}\in\mathbb{R}^{n\times m}$ satisfies $\mathcal{R}(\hat{\mathcal{B}}) = \mathcal{R}(\hat{\mathcal{A}}^\mathsf{T})$,
then \eqref{eq:scaled_ls} is equivalent to
\begin{equation}\label{eq:scaled_AB}
\hat{\mathcal{A}}\hat{\mathcal{B}}\hat{\boldsymbol{z}} = \hat{\boldsymbol{f}}, \qquad \Delta\boldsymbol{w} = \hat{\mathcal{B}}\hat{\boldsymbol{z}}
\end{equation}
for all $\hat{\boldsymbol{f}}\in \mathcal{R}(\hat{\mathcal{A}})$.
The methods discussed earlier can be applied to \eqref{eq:scaled_AB}.
In the NE-(S)SOR inner iterations, one has to compute $\|\hat{\boldsymbol{\alpha}}_i\|_2$, the norm of the $i$th row of $\hat{\mathcal{A}}$. 
However, this can be omitted if the $i$th diagonal element of $\mathcal{D}$ is chosen as the norm of the $i$th row of $\mathcal{A}$, that is, $\mathcal{D}(i,i):=\|\boldsymbol{\alpha}_i\|_2, \; i = 1, \dots, m$.
With this choice, the matrix $\hat{\mathcal{A}}$ has unit row norm
$\|\hat{\boldsymbol{\alpha}}_i\|_2 = 1, \;  i = 1, \dots, m$.
Hence, we do not have to compute the norms $\|\hat{\boldsymbol{\alpha}}_i\|_2$ inside the NE-(S)SOR inner iterations if we compute the norms $\|\boldsymbol{\alpha}_i\|_2$ for the construction of the scaling matrix $\mathcal{D}$. 
The row-scaling scheme does not incur extra CPU time. We observe in the numerical results that this scheme improves the convergence of the Krylov subspace methods.

CGNE and MRNE preconditioned by inner iterations applied to a scaled linear system $\mathcal{D}^{-1} \! \mathcal{A} \Delta \boldsymbol{w} = \mathcal{D}^{-1} \boldsymbol{f}$ are equivalent to CG and MINRES applied to $\mathcal{D}^{-1} \mathcal{A} \mathcal{A}^\mathsf{T} C^{\langle\ell\rangle} \mathcal{D} \boldsymbol{v} = \boldsymbol{f}$, $\Delta \boldsymbol{w} = \mathcal{A}^\mathsf{T} C^{\langle\ell\rangle} \mathcal{D} \boldsymbol{v}$, respectively, and hence determine the minimum-norm solution of $\mathcal{A} \Delta \boldsymbol{w} = \boldsymbol{f}$ for all $\boldsymbol{f} \in \mathcal{R}(\mathcal{A})$ and for all $\Delta \boldsymbol{w}_0 \in \mathbb{R}^n$ if $C^{\ell}$ is SPD.
Now we give conditions under which AB-GMRES preconditioned by inner iterations applied to a scaled linear system $\mathcal{D}^{-1} \! \mathcal{A} \Delta \boldsymbol{w} = \mathcal{D}^{-1} \boldsymbol{f}$ determines the minimum-norm solution of the unscaled one $\mathcal{A} \Delta \boldsymbol{w} = \boldsymbol{f}$.

\begin{lemma}\label{lm:AB}
If $\mathcal{R}(\mathcal{B}) = \mathcal{R}(\mathcal{A}^\mathsf{T})$ and $\mathcal{D} \in \mathbb{R}^{m \times m}$ is nonsingular, then AB-GMRES applied to $\mathcal{D}^{-1} \! \mathcal{A} \Delta \boldsymbol{w} = \mathcal{D}^{-1} \boldsymbol{f}$ determines the solution of $\min \| \Delta \boldsymbol{w} \|_2$, subject to $\mathcal{A} \Delta \boldsymbol{w} = \boldsymbol{f}$ without breakdown for all $\boldsymbol{f} \in \mathcal{R}(\mathcal{A})$ and for all $\Delta \boldsymbol{w}_0 \in \mathbb{R}^n$ if and only if $\mathcal{N}(\mathcal{B}) \cap \mathcal{R}(\mathcal{D}^{-1} \! \mathcal{A}) = \lbrace \boldsymbol{0} \rbrace$.
\end{lemma}
\begin{proof}
Since $\mathcal{R}(\mathcal{B}) = \mathcal{R}(\mathcal{A}^\mathsf{T})$ gives $\mathcal{R}(\mathcal{D}^{-1} \! \mathcal{A} \mathcal{B}) = \mathcal{R}(\mathcal{D}^{-1} \! \mathcal{A} \mathcal{A}^\mathsf{T}) = \mathcal{R}(\mathcal{D}^{-1} \! \mathcal{A})$, the equality\break $\min_{\boldsymbol{u} \in \mathbb{R}^m} \| \mathcal{D}^{-1} (\boldsymbol{f} - \mathcal{A} \mathcal{B} \boldsymbol{u}) \|_2 = \min_{\Delta \boldsymbol{w} \in \mathbb{R}^n} \| \mathcal{D}^{-1} (\boldsymbol{f} - \mathcal{A} \Delta \boldsymbol{w}) \|_2$ holds for all $\boldsymbol{f} \in \mathbb{R}^m$ \cite[Theorem 3.1]{HayamiYinIto2010}.
AB-GMRES applied to $\mathcal{D}^{-1} \! \mathcal{A} \Delta \boldsymbol{w} = \mathcal{D}^{-1} \boldsymbol{f}$ determines the $k$th iterate $\Delta \boldsymbol{w}_k$ by minimizing $\| \mathcal{D} (\boldsymbol{f} - \mathcal{A} \Delta \boldsymbol{w}) \|_2$ over the space $\Delta \boldsymbol{w}_0 + \mathcal{K}_k (\mathcal{D}^{-1} \! \mathcal{A} \mathcal{B}, \mathcal{D}^{-1} \boldsymbol{g}_0)$, and thus determines the solution of $\min \| \Delta \boldsymbol{w} \|_2$, subject to $\mathcal{D}^{-1} \! \mathcal{A} \Delta \boldsymbol{w} = \mathcal{D}^{-1} \boldsymbol{f}$ without breakdown for all $\boldsymbol{f} \in \mathcal{R}(\mathcal{A})$ and for all $\Delta \boldsymbol{w}_0 \in \mathbb{R}^n$ if and only if $\mathcal{N}(\mathcal{D}^{-1} \! \mathcal{A} \mathcal{B}) \cap \mathcal{R}(\mathcal{D}^{-1} \! \mathcal{A} \mathcal{B}) = \lbrace \boldsymbol{0} \rbrace$ \cite[Theorem 5.2]{MorikuniHayami2015}, which reduces to $\mathcal{R}(\mathcal{D}^{-1} \! \mathcal{A}) \cap \mathcal{N}(\mathcal{B}) = \lbrace \boldsymbol{0} \rbrace$ from $\mathcal{N}(\mathcal{D}^{-1} \! \mathcal{A} \mathcal{B}) = \mathcal{R}(\mathcal{B}^\mathsf{T} \! \mathcal{A}^\mathsf{T} \! \mathcal{D}^{-\mathsf{T}})^\perp = \mathcal{R}(\mathcal{B}^\mathsf{T} \! \mathcal{A}^\mathsf{T})^\perp = \mathcal{R}(\mathcal{B}^\mathsf{T} \mathcal{B})^\perp = \mathcal{R}(\mathcal{B}^\mathsf{T})^\perp = \mathcal{N}(\mathcal{B})$.
\end{proof}

\begin{theorem}
If $\mathcal{D} \in \mathbb{R}^{m \times m}$ is nonsingular and the inner-iteration matrix is semiconvergent, then AB-GMRES preconditioned by the inner iterations applied to $\mathcal{D}^{-1} \! \mathcal{A} \Delta \boldsymbol{w} = \mathcal{D}^{-1} \boldsymbol{f}$ determines the solution of $\min \| \Delta \boldsymbol{w} \|_2$, subject to $\mathcal{A} \Delta \boldsymbol{w} = \boldsymbol{f}$ without breakdown for all $\boldsymbol{f} \in \mathcal{R}(\mathcal{A})$ and for all $\Delta \boldsymbol{w}_0 \in \mathbb{R}^n$.
\end{theorem}
\begin{proof}
From Lemma \ref{lm:AB}, it is sufficient to show that $\mathcal{R}(\mathcal{B}) =\mathcal{R}(\mathcal{A}^\mathsf{T})$ and \break $\mathcal{N}(\mathcal{D}^{-1} \! \mathcal{A} \mathcal{B}) \cap \mathcal{R}(\mathcal{D}^{-1} \! \mathcal{A} \mathcal{B}) = \lbrace \boldsymbol{0} \rbrace$.
Since $\mathcal{D}^{-1} \! M \mathcal{D}^{-\mathsf{T}} = \mathcal{D}^{-1} (\mathcal{A} \mathcal{A}^\mathsf{T} - N) \mathcal{D}^{-\mathsf{T}}$ is the splitting matrix of $\mathcal{D}^{-1}  \mathcal{A} \mathcal{A}^\mathsf{T} \mathcal{D}^{-\mathsf{T}}$ for the inner iterations, the inner-iteration matrix is $\mathcal{D}^\mathsf{T} \! H \mathcal{D}^{-\mathsf{T}}$.
Hence, the inner-iteration preconditioning matrix $\mathcal{B} = \mathcal{A}^\mathsf{T} C^{\langle\ell\rangle} \mathcal{D}$ satisfies $\mathcal{R}(\mathcal{B}) = \mathcal{R}(\mathcal{A}^\mathsf{T})$ \cite[Lemma 4.5]{MorikuniHayami2015}.
On the other hand, $\mathcal{D}^{-1} \! \mathcal{A} \mathcal{B} = \mathcal{D}^{-1} \! M (I - H^{\ell}) (\mathcal{D}^{-1} \! M)^{-1}$ satisfies $\mathcal{N}(\mathcal{D}^{-1} \! \mathcal{A} \mathcal{B}) \cap \mathcal{R}(\mathcal{D}^{-1} \! \mathcal{A} \mathcal{B}) = \lbrace \boldsymbol{0} \rbrace$ \cite[Lemmas 4.3, 4.4]{MorikuniHayami2015}.
\end{proof}

\section{Numerical experiments}\label{sec:lp_NumericalExperiments}
In this section, we compare the performance of the interior-point method based on the iterative solvers with the standard interior-point programs.  
We also developed an efficient direct solver coded in C to compare with the iterative solvers.  
For the sake of completeness, we briefly describe our direct solver first.
\subsection{Direct solver for the normal equations}
To deal with the rank-deficiency,  
we used a strategy that is similar to the Cholesky-Infinity modification scheme introduced in the LIPSOL solver \cite{Zhang1998}.
However, instead of penalizing the elements that are close to zero, we removed them and solved the reduced system.
We implemented this modification by an LDLT decomposition. 
We used the \textsc{Matlab} built-in function {\texttt{chol}} to detect whether the matrix is symmetric positive definite.
We used the {\texttt{ldlchol}} from \textsc{CSparse} package version 3.1.0 \cite{Davis2014} when the matrix was symmetric positive definite, and we turned to the \textsc{Matlab} built-in solver {\texttt{ldl}} for the semidefinite cases which uses MA57 \cite{Duff2004}.

We explain the implementation by an example where $\mathcal{A}\mathcal{A}^\mathsf{T} \in \mathbb{R}^{3\times 3}$. 
For matrix $\mathcal{A}\mathcal{A}^\mathsf{T}$, LDLT decomposition gives
\begin{equation*}
\mathcal{A}\mathcal{A}^\mathsf{T} = L G L^\mathsf{T}
=\begin{bmatrix} 1 & 0 & 0\\ l_{21} & 1 & 0 \\ l_{31} & l_{32} & 1 \end{bmatrix} \begin{bmatrix}  g_1 & 0 &0 \\0 &g_2 & 0\\0&0& g_3 \end{bmatrix} \begin{bmatrix} 1 & l_{21} &  l_{31}\\0 & 1 &  l_{32} \\ 0 & 0 & 1\end{bmatrix}. 
\end{equation*}
Correspondingly, we partition 
$\Delta\boldsymbol{y}= [ \Delta y_1,\Delta y_2,\Delta y_3 ]^\mathsf{T}$ and $\boldsymbol{f}=[ f_1, f_2, f_3]^\mathsf{T}$.
Assuming that the diagonal element $g_2$ is close to zero, we let
$\tilde{L}:= \bigl[\begin{smallmatrix} 1 & 0\\l_{31} & 1 \end{smallmatrix}\bigr]$, $\tilde{G}:= \bigl[\begin{smallmatrix} g_1& 0\\0 & g_3 \end{smallmatrix}\bigr]$, $\tilde{\boldsymbol{f}} = [f_1, f_3]^\mathsf{T}$, $\tilde{\Delta\boldsymbol{y}} = [\Delta y_1, \Delta y_3]^\mathsf{T}$,
and solve
\begin{equation*}
\tilde{L}\tilde{G}^{1/2}\left( (\tilde{L}\tilde{G}^{1/2})^\mathsf{T} \tilde{\Delta\boldsymbol{y}}\right) = \tilde{\boldsymbol{f}},
\end{equation*}
using forward and backward substitutions. 
The solution is then given by $\Delta\boldsymbol{y} = [\Delta y_1, 0, \Delta y_3]^\mathsf{T}$.
\FloatBarrier
\subsection{Implementation specifications}
In this section, we describe our numerical experiments. 

The initial solution for the interior-point method was set using the method described in LIPSOL solver \cite{Zhang1998}.
The initial solution for the Krylov subspace iterations and the inner iterations was set to zero. 

We set the maximum number of the interior-point iterations as $99$ and the stopping criterion regarding the error measure as 
\begin{equation}\label{eq:errorMeasure}
\Gamma^{(k)}\leq \epsilon_\mathrm{out} = 10^{-8},
\end{equation}
where $\Gamma^{(k)}$ is defined by \eqref{eq:errorMeasureDef}.

For the iterative solver for the linear system \eqref{eq:minnrmprob}, 
we set the maximum number of iterations for CGNE, MRNE and AB-GMRES as $m$, and relaxed it to $40,000$ for some difficult problems for CGNE and MRNE. 
We set the stopping criterion for the scaled residual as
\begin{equation*}
\| \hat{\boldsymbol{f}}-\hat{\mathcal{A}}\Delta\boldsymbol{w}^{(k)}\|_2 \leq \epsilon_\mathrm{in}\|\hat{\boldsymbol{f}}\|_2,
\end{equation*}
where $\epsilon_\mathrm{in} $ is initially $10^{-6}$ and is kept in the range $[10^{-14}, 10^{-4}]$ during the process. 
We adjusted $\epsilon_\mathrm{in} $ according to the progress of the interior-point iterations. 
We truncated the iterative solving prematurely in the early interior-point iterations, and pursued a more precise direction as the LP solution was approached. 
The progress was measured by the error measure $\Gamma^{(k)}$. Concretely, we adjusted $\epsilon_\mathrm{in} $ as 
\begin{equation*}
 \epsilon_\mathrm{in}^{(k)} = 
 \begin{cases}   \epsilon_\mathrm{in}^{(k-1)} \times 0.75 & \boldsymbol{if } \log_{10}{\Gamma^{(k)}}\in (-3, 1],\\
  \epsilon_\mathrm{in}^{(k-1)} \times 0.375 & \boldsymbol{if } \log_{10}{\Gamma^{(k)}} \in (-\infty,-3].\end{cases}
\end{equation*}
For steps where iterative solvers failed to converge within the maximum number of iterations, we adopted the iterative solution with the minimum residual norm and slightly increased the value of $\epsilon_\mathrm{in}$ by multiplying by $1.5$ which would be used in the next interior-point step.

Note that preliminary experiments were conducted with the tolerance being fixed for all the problems. 
However, further experiments showed that adjusting the parameter $\epsilon_\mathrm{in}$ with the progress towards an optimal solution worked better. This is also another advantage of using iterative solvers rather than direct solvers.

We adopt the implementation of AB-GMRES preconditioned by NE-SOR inner-iterations \cite{MorikuniHayamiCodesABMEX} with the additional row-scaling scheme (Section \ref{sec:rowscaling}).
No restarts were used for the AB-GMRES method.
The non-breakdown conditions discussed in Sections \ref{sec:CGNEMRNE}, \ref{sec:ABGMRES} are satisfied.

For the direct solver, the tolerance for dropping pivot elements close to zero was $10^{-16}$ for most of the problems; for some problems this tolerance has to be increased to $10^{-6}$ to overcome breakdown. 

The experiment was conducted on a MacBook Pro with a 2.6 GHz Intel Core i5 processor with 8 GB of random-access memory, OS X El Capitan version 10.11.2. 
The interior-point method was coded in \textsc{Matlab} R2014b and the iterative solvers including AB-GMRES (NE-SOR), CGNE (NE-SSOR), and MRNE (NE-SSOR), were coded in C and compiled as \textsc{Matlab} Executable (MEX) files accelerated with Basic Linear Algebra Subprograms (BLAS).

We compared our implementation with PDCO version 2013 \cite{saunders2002pdco} and three solvers available in \textsc{CVX} \cite{GrantBoyd2014,GrantBoyd2008}: SDPT3 version 4.0 \cite{TohToddTutuncu1999,TutuncuTohTodd2003}, SeDuMi version 1.34 \cite{TohToddTutuncu1999} and MOSEK version 7.1.0.12 \cite{mosek},
with the default interior-point stopping criterion \eqref{eq:errorMeasure}. 
Note that SDPT3, SeDuMi, and PDCO are non-commercial public domain solvers, whereas MOSEK is a commercial solver known as one of the state-of-the-art solvers. 
% PDCO was tested using its \textsc{Matlab} interface. 
PDCO provides several choices for the solvers for the interior-point steps, among which we chose the direct (Cholesky) method and the LSMR method. Although MINRES solver is another iterative solver available in PDCO, its homepage \cite{saunders2002pdco} suggests that LSMR performs better in general. Thus, we tested with LSMR. 
For PDCO parameters, we chose to suppress scaling for the original problem.
The other solvers were implemented with the CVX \textsc{Matlab} interface, and we recorded the CPU time reported in the screen output of each solver. 
However, it usually took a longer time for the CVX to finish the whole process. 
The larger the problem was, the more apparent this extra CPU time became. For example, for problem {\texttt ken\_\!\! 18}, the screen output of SeDuMi was 765.3 seconds while the total processing time was 7,615.2 seconds.

We tested on two classes of LP problems: 127 typical problems from the benchmark libraries and 13 problems arising from basis pursuit. The results are described in Section \ref{sec:standard_lp} and Section \ref{sec:bp_lp}, respectively.

\FloatBarrier
\subsection{Typical LP problems: sparse and ill-conditioned problems}\label{sec:standard_lp}
We tested 127 typical LP problems from the \textsc{Netlib}, \textsc{Qaplib} and \textsc{Mittelmann} collections in \cite{DavisHu2011}. 
Most of the problems have sparse and full-rank constraint matrix $A$ (except problems {\texttt{bore3d}} and {\texttt{cycle}}).
For the problems with $\boldsymbol{l}\leq\boldsymbol{x}\leq\boldsymbol{u},\;\boldsymbol{l}\neq\boldsymbol{0},\;\boldsymbol{u}\neq\boldsymbol{\infty}$, we transform them using the approach in LIPSOL~\cite{Zhang1998}.

The overall summary of numerical experiments on the 127 typical problems is given in Table \ref{tab:overall}.
The counts in column ``Failed'' include the case where a problem was solved at a relaxed tolerance (phrased as ``inaccurately solved'' in CVX). Column ``Expensive'' refers to the case where the interior-point iterations took more than a time limit of 20 hours.

\begin{table}[!bp]
\footnotesize
\renewcommand{\arraystretch}{1.2}
\caption{Overall performance of the solvers on $127$ testing problems.}
\begin{center}
\begin{tabular}{@{\extracolsep{\fill}}rrrr@{\extracolsep{\fill}}}
\hline
Status & Solved & Failed & Expensive\\ \hline
AB-GMRES (NE-SOR) & 123   & 2 & 2 \\
CGNE (NE-SSOR) &124 &3  &0\\
MRNE (NE-SSOR) &125 &2  &0\\
Modified Cholesky &117 &10 &0\\
SDPT3 & 76 & 46  & 5\\
SeDuMi & 104 & 23 & 0 \\
MOSEK & 127 & 0 & 0\\
PDCO (Direct) & 110  &17 & 0 \\
PDCO (LSMR) & 88  & 35 & 4\\
\hline 
\end{tabular}\label{tab:overall}
\end{center}
\end{table}

MOSEK was most stable in the sense that it solved all 127 problems, and MRNE (NE-SSOR) came next with only two failures with the {\textsc{Netlib}} problems {\texttt{greanbea}} and {\texttt{greanbeb}}.
CGNE (NE-SSOR) method solved almost all the problems that MRNE (NE-SSOR) solved, except for the 
largest \textsc{Qaplib} problem, which was solved to a slightly larger tolerance level of $10^{-7}$.
AB-GMRES (NE-SOR) was also very stable and solved the problems accurately enough. However, it took longer than 20 hours for two problems
that have 105,127 and 16,675 equations,
% that have 154,699 and 23,541 unknowns, 
respectively, although it succeeded in solving larger problems such as {\texttt{pds-80}}.
The other solvers were less stable.
The modified Cholesky solver and PDCO (Direct) solved $92\%$ and $87\%$ of the problems, respectively,
although they were faster than the other solvers for the problems that they could successfully solve.
PDCO (LSMR) solved $69\%$ problems and was slower than the proposed solvers. The reason could be that it does not use preconditioners.
SDPT3 solved $60\%$ and SeDuMi $82\%$ of the problems.
Here we should mention that SeDuMi and SDPT3 are designed for LP, SDP, and SOCP, while our code is (currently) tuned solely for LP.

Note that MOSEK solver uses a multi-corrector interior-point method \cite{Gondzio1996} while our implementation is a single corrector (i.e., predictor-corrector) method.
This led to different numbers of interior-point iterations as shown in the tables.
Thus, there is still room for improvement in the efficiency of our solver based on iterative solvers if a more elaborately tuned interior-point framework such as the one in MOSEK is adopted.

\begin{figure}[!bt]
   \subfloat[\textsc{Netlib} problems.]{\includegraphics[width=.47\textwidth]{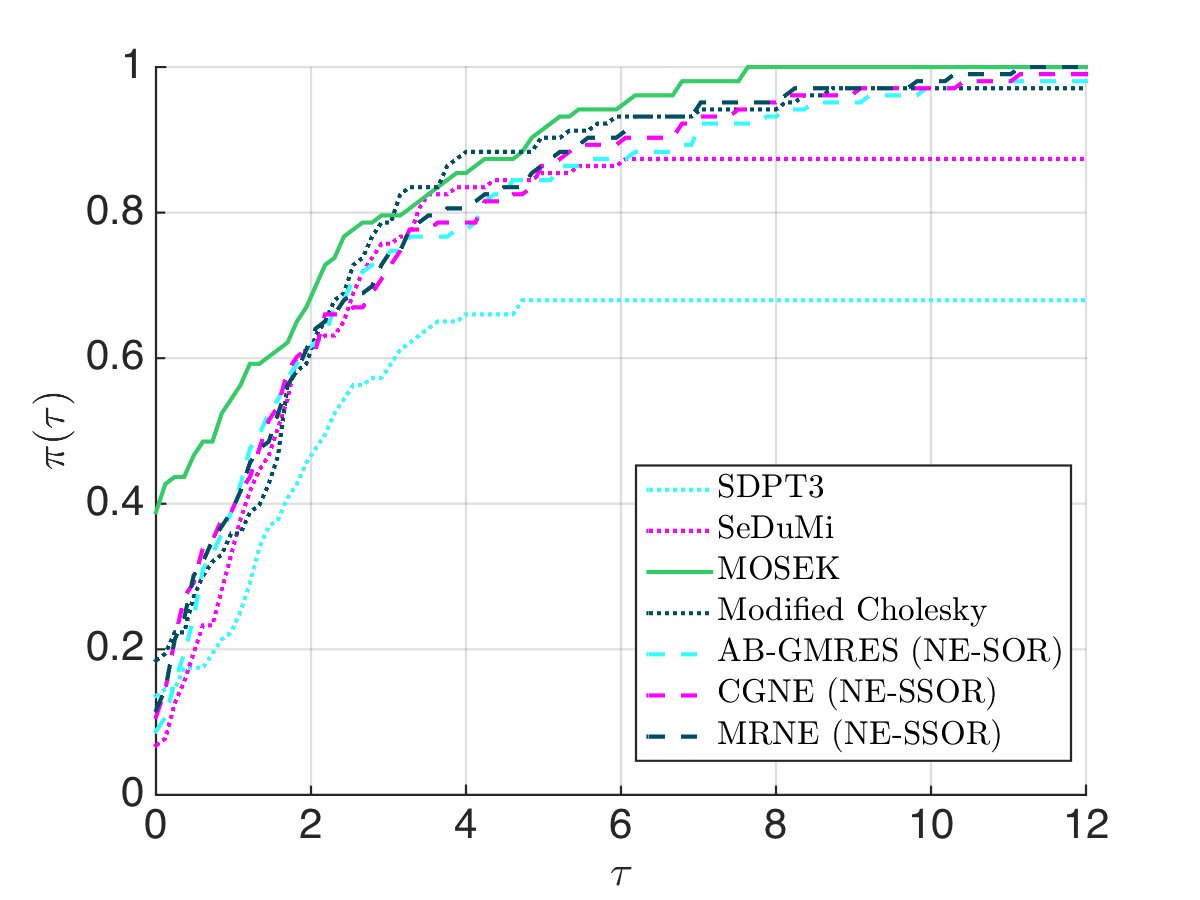}\label{fig:profile_netlib}}\hfill
   \subfloat[\textsc{Qaplib} problems.]{\includegraphics[width=.47\textwidth]{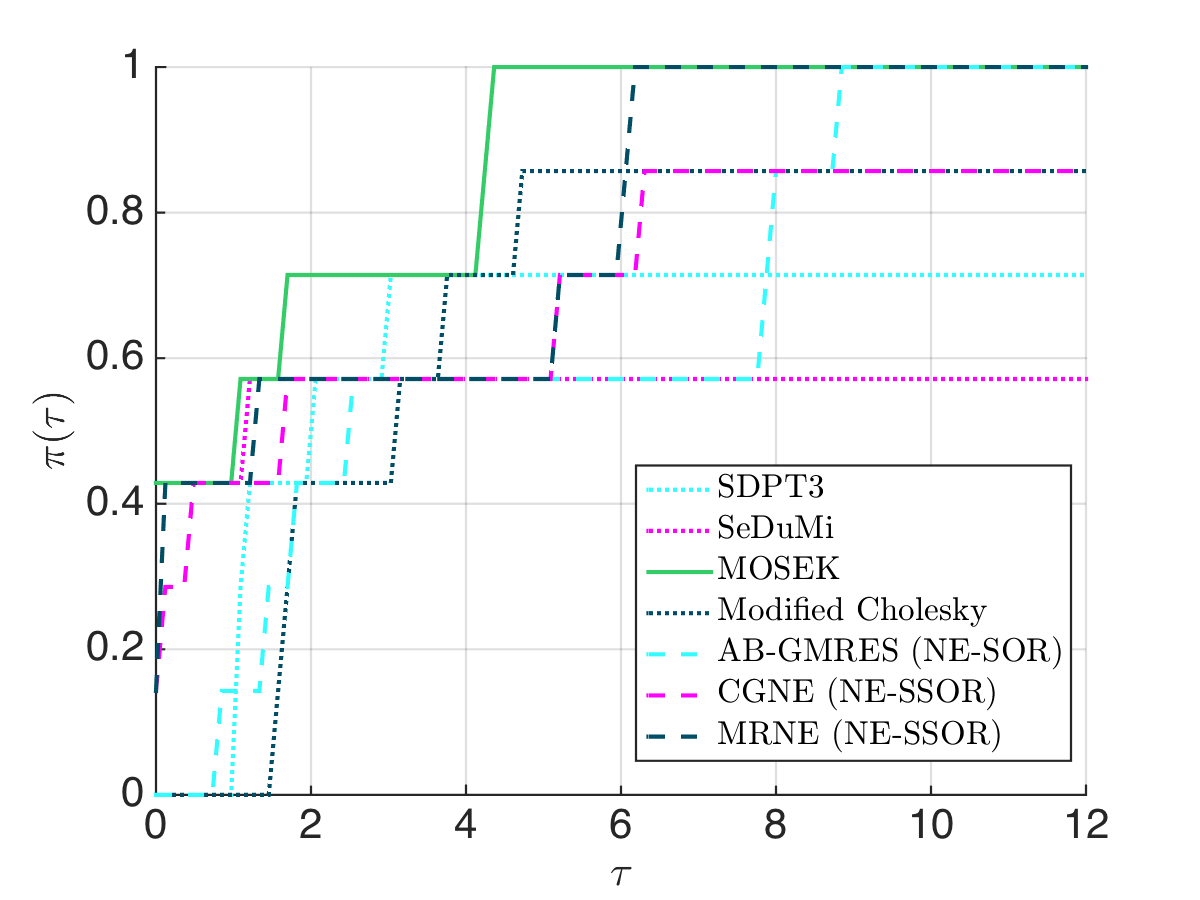}\label{fig:profile_qaplib}}\\
   \subfloat[\textsc{Mittelmann} problems.]{\includegraphics[width=.47\textwidth]{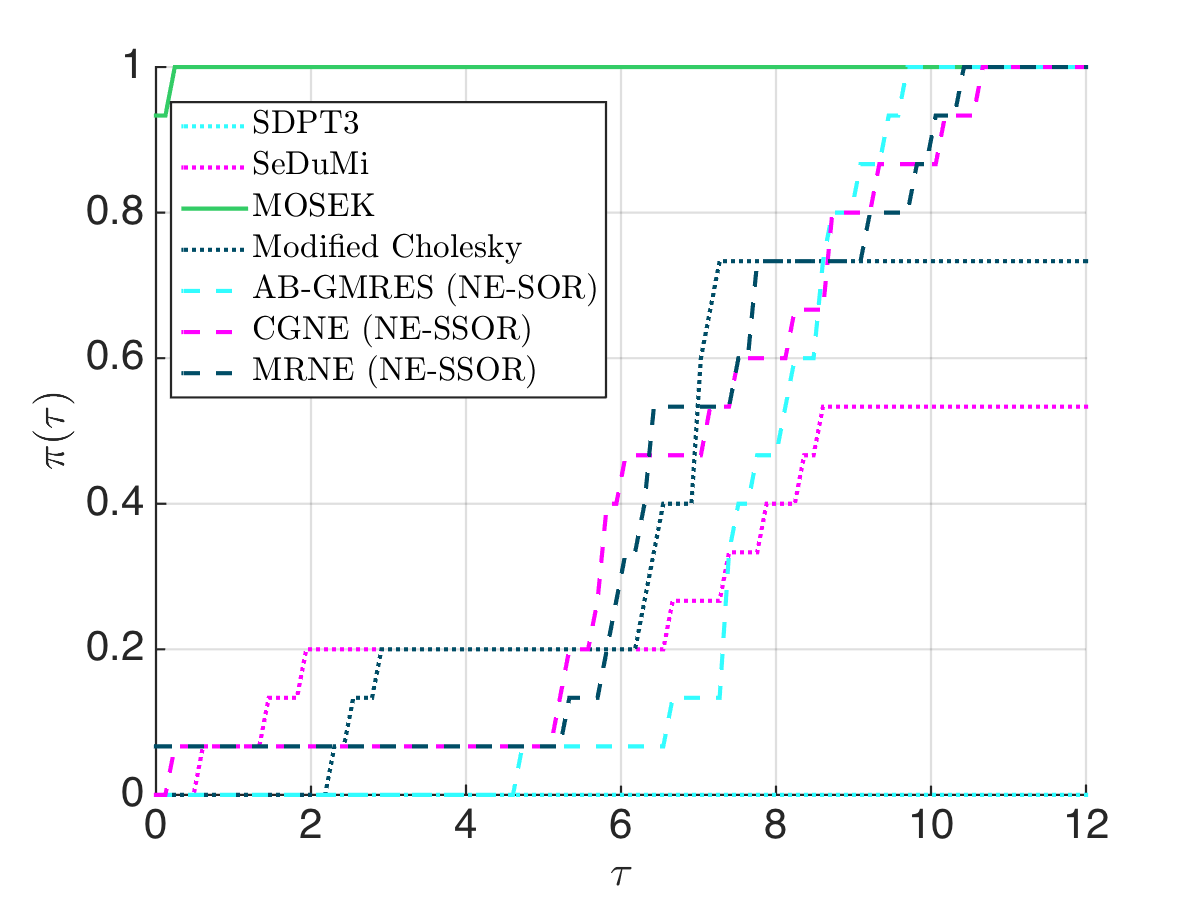}\label{fig:profile_mittelmann}}\hfill
    \subfloat[All the problems.]{\includegraphics[width=.47\textwidth]{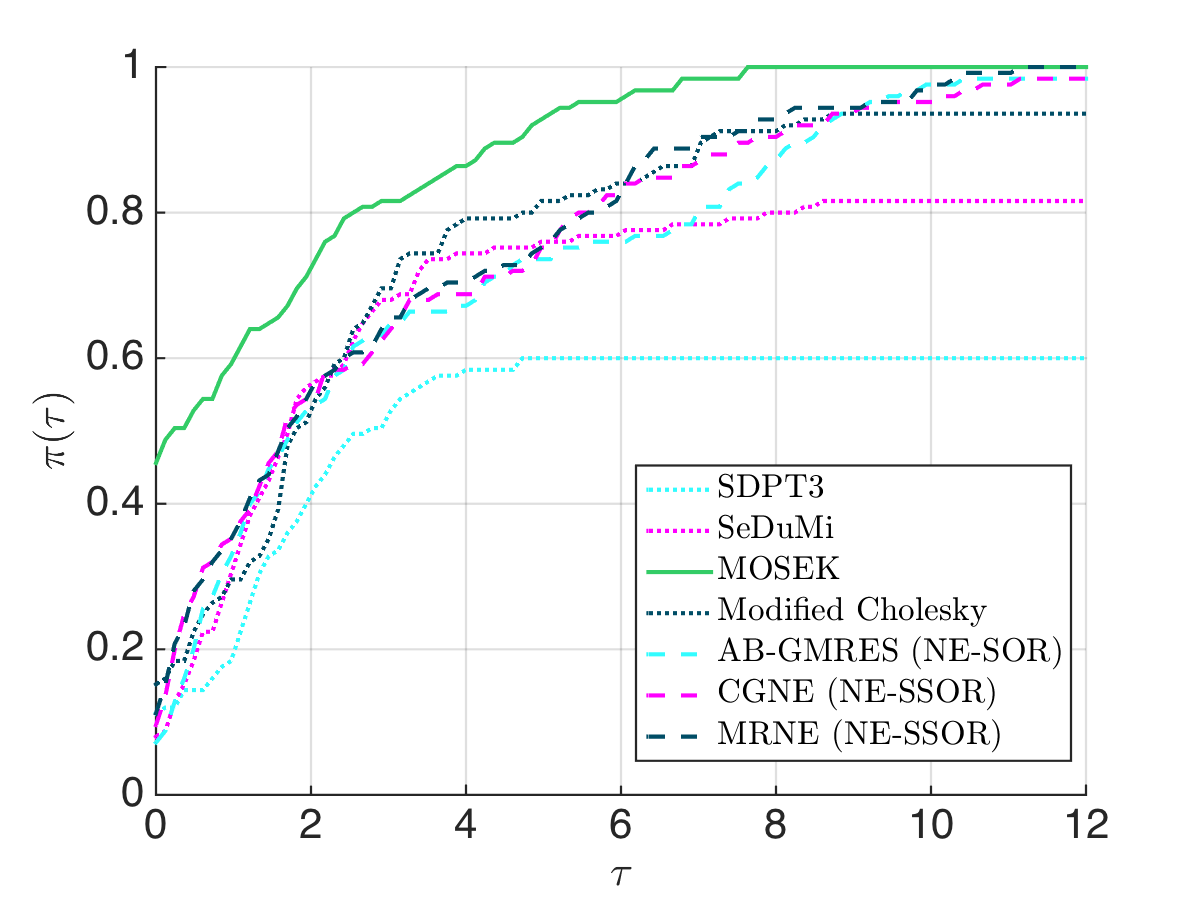}\label{fig:profile_all}}
    \caption{Dolan-Mor{\'e} profiles comparing the CPU time costs for the proposed solvers, public domain and commercial solvers.\label{fig:profile}}
\end{figure}
\begin{figure}[!bt]
    \subfloat[\textsc{Netlib} problems.]{\includegraphics[width=.47\textwidth]{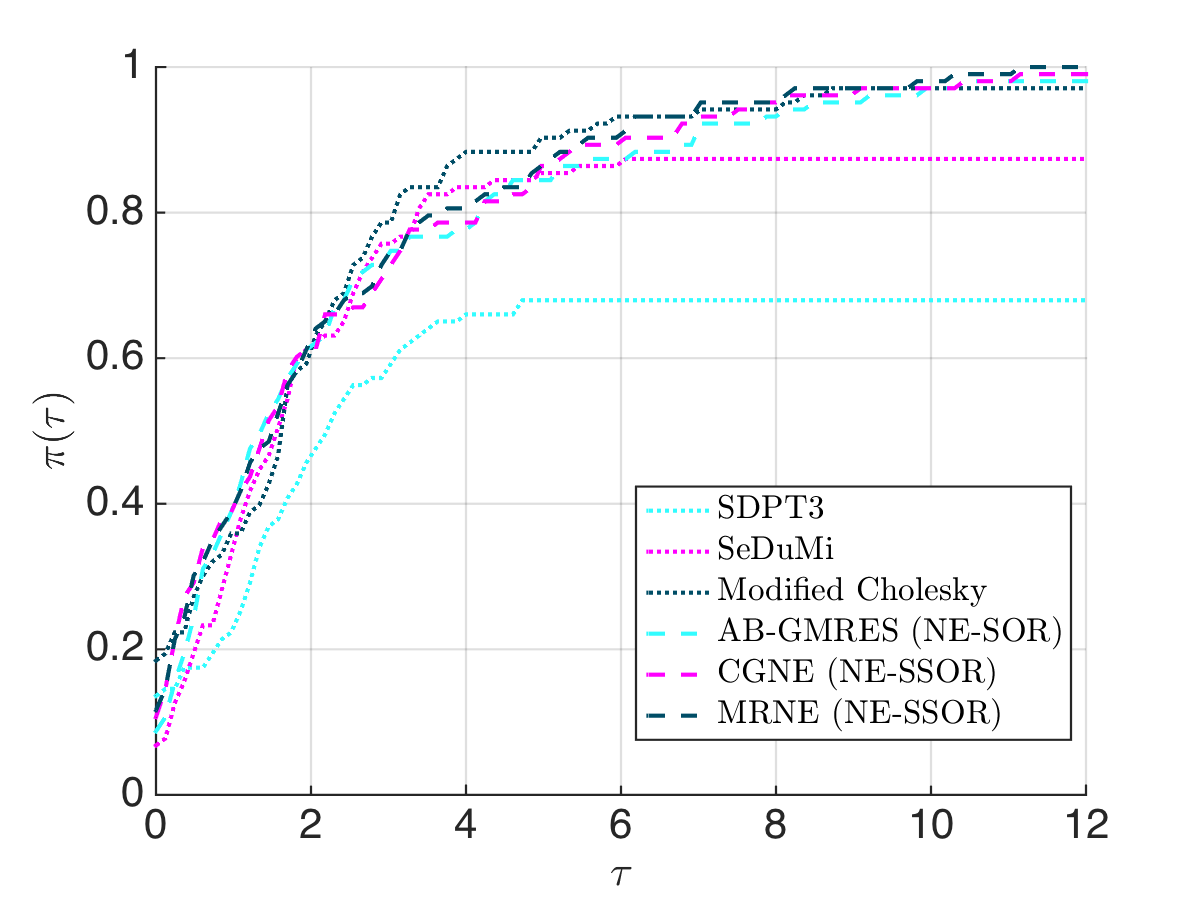}\label{fig:profile_netlib_noMosek}}\hfill
    \subfloat[\textsc{Qaplib} problems.]{\includegraphics[width=.47\textwidth]{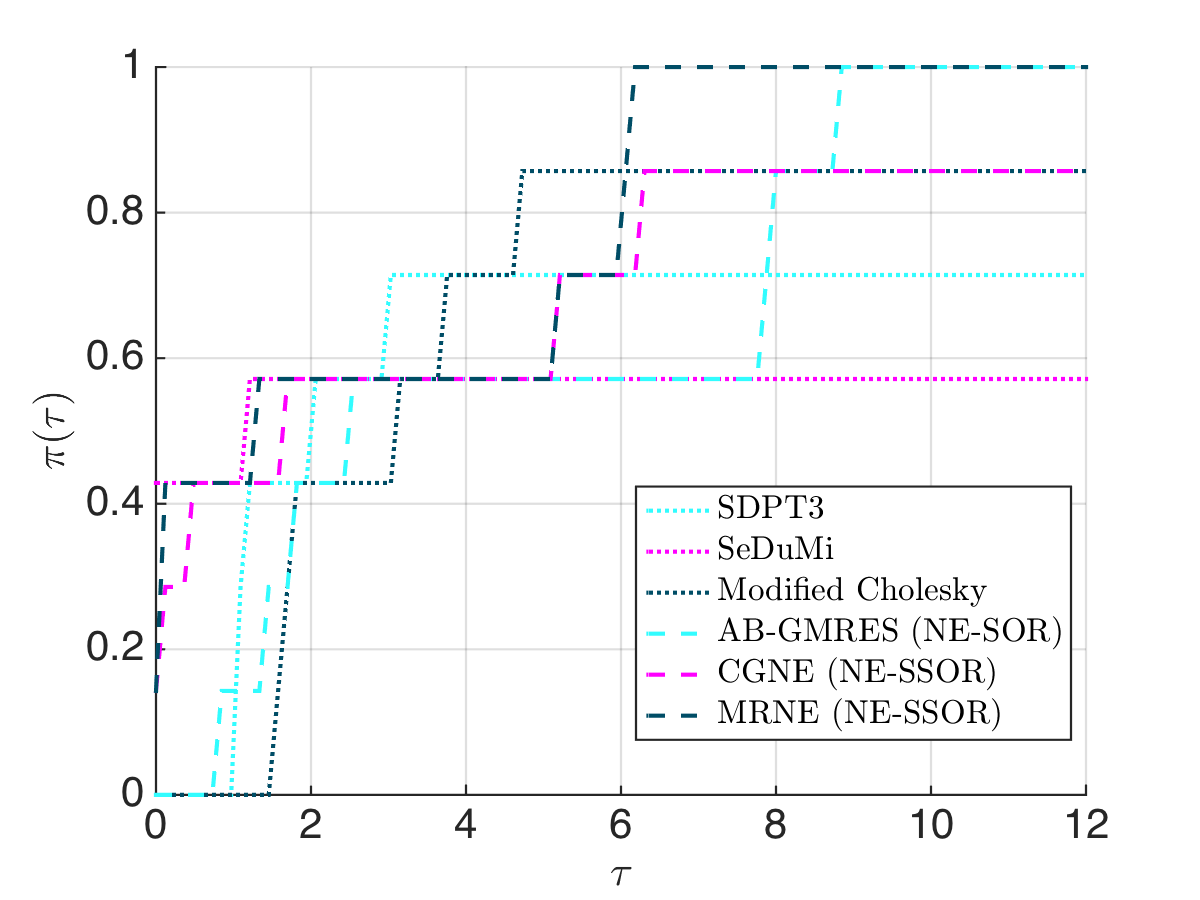}\label{fig:profile_qaplib_noMosek}}\\
    \subfloat[\textsc{Mittelmann} problems.]{\includegraphics[width=.47\textwidth]{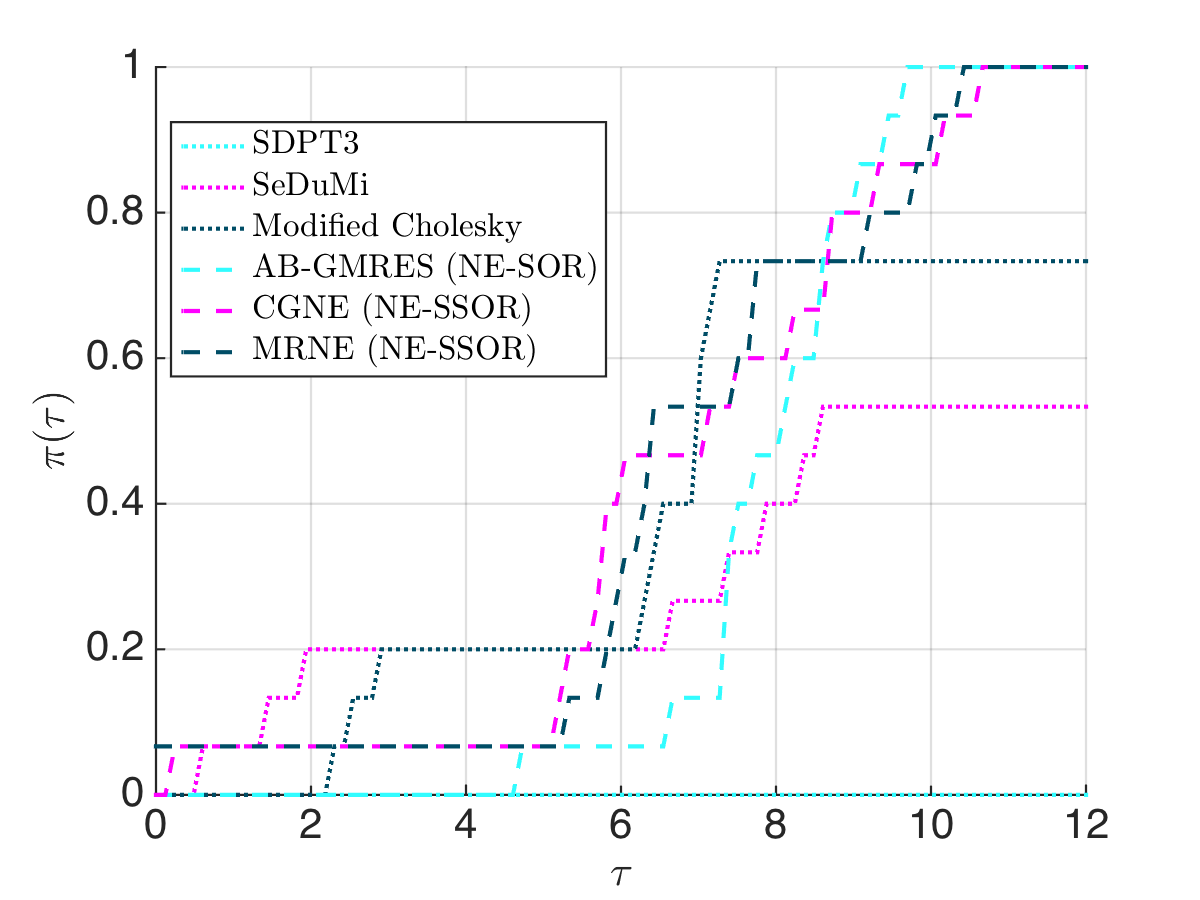}\label{fig:profile_mittelmann_noMosek}}\hfill
    \subfloat[All the problems.]{\includegraphics[width=.47\textwidth]{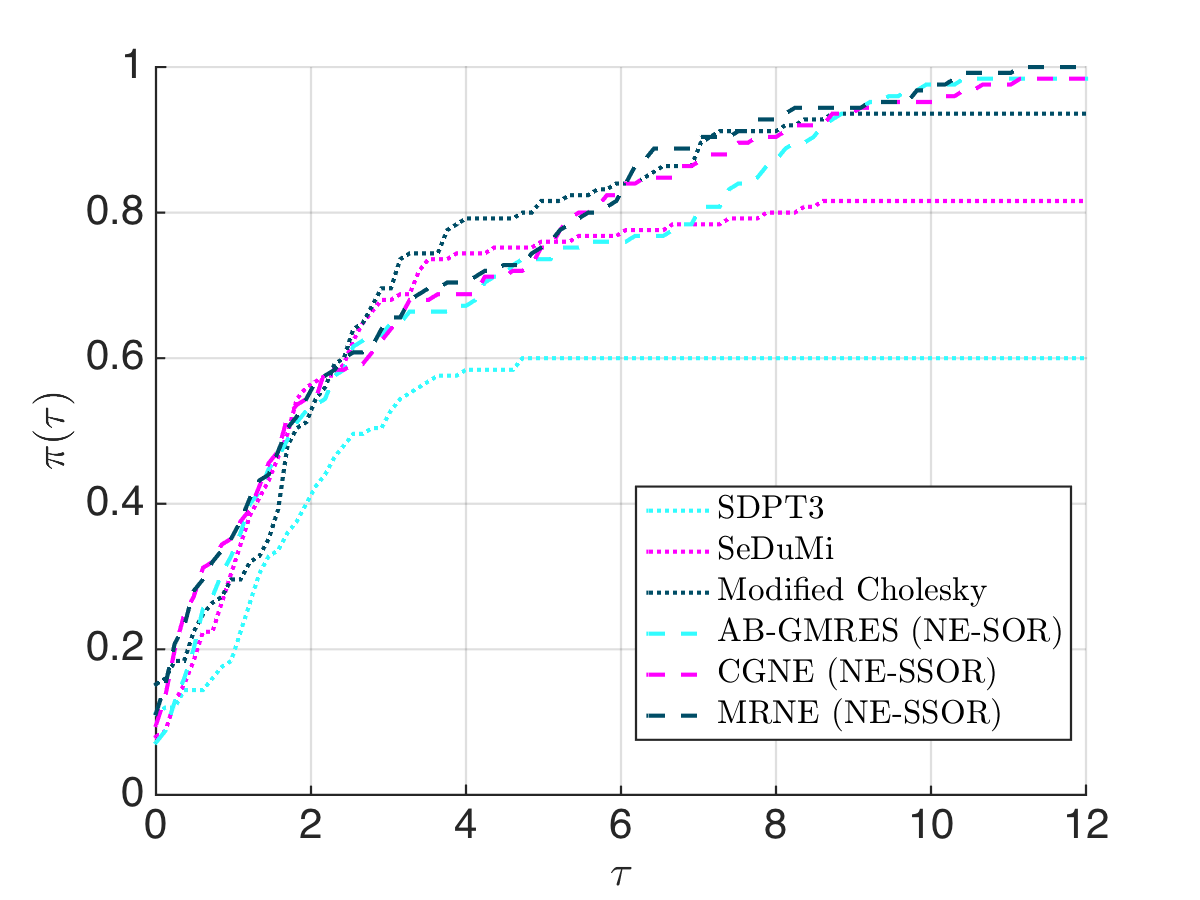}}
\caption{Dolan-Mor{\'e} profiles comparing the CPU time costs for the proposed solvers and public domain solvers.\label{fig:profile_noMosek}}
\end{figure}

In order to show the trends of performance, we use the Dolan-Mor{\'e} performance profiles \cite{DolanMore2002} in Figures \ref{fig:profile} and \ref{fig:profile_noMosek}, with $\pi(\tau):=P(\log_2 r_{ps}\leq\tau)$ the proportion of problems for which $\log_2$-scaled performance ratio is at most $\tau$, where $r_{ps}:=t_{ps}/t^*_\mathrm{p}$, $t_{ps}$ is the CPU time for solver $s$ to solve problem $p$, and $t^*_\mathrm{p}$ is the minimal CPU time for problem $p$.
Figure \ref{fig:profile} includes the commercial solver MOSEK while Figure \ref{fig:profile_noMosek} does not. 
Note that the generation of Figure \ref{fig:profile_noMosek} is not by simply removing the curve of MOSEK from Figure \ref{fig:profile}, but rather removing the profile of MOSEK from the comparison dataset and thus changing the minimum CPU time cost for each problem.
The comparison indicates that the iterative solvers, although slower than the commercial solver MOSEK in some cases, were often able to solve the problems to the designated accuracy.

In Tables \ref{tab:netlib1}, \ref{tab:qaplib}, and \ref{tab:mittle}, we give the following information:
\begin{enumerate}
\item the name of the problem and the size $(m, n)$ of the constraint matrix,
\item the number of interior-point iterations required for convergence,
\item CPU time for the entire computation in seconds. For the cases shorter than $3,000$ seconds, CPU time is taken as an average over 10 measurements. 
In each row, we indicate in red boldface and blue underline the fastest and second fastest solvers in CPU time, respectively.
\end{enumerate}

% \FloatBarrier

Besides the statistics, we also use the following notation:
\begin{itemize}
\item[\dag] inaccurately solved, i.e., the value of $\epsilon_\mathrm{out}$ was relaxed to a larger level. 
In the column ``Iter'', we provide extra information \dag$_a$ at the stopping point: 
for our solvers, $a=\floor{\log_{10}\Gamma^{(k)}}$,
where $\floor{\cdot}$ is the floor function; for CVX solvers, $a=\floor{\log_{10}\mu}$ as provided in the CVX output; PDCO solvers do not provide this information, thus they are not given;
\item[$\texttt{f}$] the interior-point iterations diverged;
\item[$\texttt{t}$] the iterations took longer than 20 hours.
\end{itemize}

Note that all zero rows and columns of the constraint matrix $A$ were removed beforehand.
The problems marked with ${\#}$ are with rank-deficient $A$ even after this preprocessing. 
For these problems we put $\mathrm{rank}(A)$ in brackets after $m$, which is computed using the \textsc{Matlab} function $\texttt{sprank}$.

In order to give an idea of the typical differences between methods, we present the interior-point convergence curves for problem {\texttt{ken\_\!\! 13}}.
The problem has a constraint matrix $A\in\mathbb{R}^{28,632\times 42,659}$ with full row rank and $97,246$ nonzero elements.

Different aspects of the performance of the four solvers are displayed in Figure \ref{fig:ken13}.
The red dotted line with diamond markers represents the quantity related to AB-GMRES (NE-SOR), the blue with downward-pointing triangle CGNE (NE-SSOR), the yellow with asterisk MRNE (NE-SSOR), and the dark green with plus sign the modified Cholesky solver. 
Note that for this problem {\texttt{ken\_\!\! 13}}, the modified Cholesky solver became numerically inaccurate at the last step and it broke down if the default dropping tolerance was used. Thus, we increased it to $10^{-6}$.  

Figure \ref{fig:cond} shows $\kappa(\mathcal{A}\mathcal{A}^\mathsf{T})$ in $\log_{10}$ scale. 
It verifies the claim that the least squares problem becomes increasingly ill-conditioned at the final steps in the interior-point process: $\kappa(\mathcal{A}\mathcal{A}^\mathsf{T})$ started from around $10^{20}$ and increased to $10^{80}$ at the last 3-5 steps.
Figure \ref{fig:mu} shows the convergence curve of the duality measure $\mu$ in $\log_{10}$ scale. 
The $\mu$ drops below the tolerance and the stopping criterion is satisfied.
Although it is not shown in the figure, we found that the interior-point method with modified
Cholesky with the default value of the dropping tolerance $10^{-16}$ stagnated for $\mu \simeq 10^{-4}$.
Comparing with Figure \ref{fig:cond}, it is observed that the solvers started to behave differently as $\kappa(\mathcal{A}\mathcal{A}^\mathsf{T})$ increased sharply. 

\afterpage{
\vspace{+1cm}
\begin{landscape}
\begin{table}[htbp]
\centering
\footnotesize
\renewcommand\tabcolsep{3pt}
\caption{Experiments on \textsc{Netlib} problems.
In each row, red boldface and blue underline denote the fastest and second fastest solvers in CPU time, respectively.}
\renewcommand{\arraystretch}{1.1}
\begin{tabular}{@{}rrrrrrrrrrrrrrrrrrrrr@{}}
 \hline 
& && \multicolumn{2}{c}{AB-GMRES}  &  \multicolumn{2}{c}{CGNE} &  \multicolumn{2}{c}{MRNE} &  \multicolumn{2}{c}{Modified}  & \multicolumn{2}{c}{SDPT3 }  & \multicolumn{2}{c}{SeDuMi } & \multicolumn{2}{c}{MOSEK} & \multicolumn{2}{c}{PDCO}& \multicolumn{2}{c}{PDCO}\\ 

& && \multicolumn{2}{c}{(NE-SOR)}  & \multicolumn{2}{c}{(NE-SSOR)}  & \multicolumn{2}{c}{(NE-SSOR)}  & \multicolumn{2}{c}{Cholesky} & \multicolumn{2}{c}{} & \multicolumn{2}{c}{} & \multicolumn{2}{c}{} & \multicolumn{2}{c}{Direct} & \multicolumn{2}{c}{LSMR}\\

Problem & $m$ & $ n$   &  Iter & Time   & Iter & Time & Iter & Time & Iter & Time     & Iter & Time      & Iter  & Time   & Iter & Time & Iter & Time& Iter & Time\\ \hline
25fv47    & 821 &1,876  & 25     & 4.62     & 25            & 5.00            & 25             & 4.60  & 25   & 3.67    
                  & 59     & 2.50        & 29        & \color{blue}\underline{2.30}      & 26             & 3.90  & 45 & \bf{\bf\color{red}1.92} &   48 & 20.60\\
% ------------------------------------------------------------------------------------------------------------------------          
adlittle  & 56 & 138 & 12     & \bf{\bf\color{red}0.03}     & 13            & \bf{\bf\color{red}0.03}            & 13             & \color{blue}\underline{0.05}           & 12   & 0.09         
          & 16     & 0.16     & 14            & 0.10            & 10             & 1.98  & 25 & 0.15&   25 & 1.60\\
% ------------------------------------------------------------------------------------------------------------------------          
afiro     & 27 & 51 & 8      & \color{blue}\underline{0.02}     & 8             & \bf{\bf\color{red}0.01}            & 8              & \bf{\bf\color{red}0.01}           & 8    & 0.03          
          & 11     & 0.11     & 7             & 0.10            & 9              & 1.91  & 18 & 0.11&   18 & 0.48 \\
\hline
\end{tabular}\label{tab:netlib1}
\end{table}
\end{landscape}

\begin{landscape}
\begin{table}[!bp]
\ContinuedFloat
\centering
\footnotesize
\renewcommand\tabcolsep{2.5pt}
\caption{(cont.) Experiments on \textsc{Netlib} problems.
In each row, red boldface and blue underline denote the fastest and second fastest solvers in CPU time, respectively.}
\renewcommand{\arraystretch}{1.1}
\begin{tabular}{@{}rrrrrrrrrrrrrrrrrrrrr@{}}
 \hline 
& && \multicolumn{2}{c}{AB-GMRES}  &  \multicolumn{2}{c}{CGNE} &  \multicolumn{2}{c}{MRNE} &  \multicolumn{2}{c}{Modified}  & \multicolumn{2}{c}{SDPT3 }  & \multicolumn{2}{c}{SeDuMi } & \multicolumn{2}{c}{MOSEK} & \multicolumn{2}{c}{PDCO}& \multicolumn{2}{c}{PDCO}\\ 

& && \multicolumn{2}{c}{(NE-SOR)}  & \multicolumn{2}{c}{(NE-SSOR)}  & \multicolumn{2}{c}{(NE-SSOR)}  & \multicolumn{2}{c}{Cholesky} & \multicolumn{2}{c}{} & \multicolumn{2}{c}{} & \multicolumn{2}{c}{} & \multicolumn{2}{c}{Direct} & \multicolumn{2}{c}{LSMR}\\

Problem & $m$  & $n$   &  Iter  & Time   & Iter & Time & Iter & Time & Iter & Time     & Iter & Time      & Iter  & Time   & Iter & Time &Iter &Time& Iter & Time\\ \hline
% name    | m     | ABGMRES          | CGNE                             | MRNE                             | Cholesky
%                 | SDPT3            | SeDuMi                           | MOSEK
etamacro  & 400 & 816  & 30     & \color{blue}\underline{1.23}     & 31            & 1.58            & 31             & 1.43            & 30    & 2.30
          & \texttt{f}      & \texttt{f}        & 30            & 2.80            & 20             & 2.82    &38 & \bf\color{red}0.87 &   38 & 3.88 \\
% ------------------------------------------------------------------------------------------------------------------------  
fffff800  & 524 & 1,028  & 32     & 4.11    & 30            & 6.29            & 33             & 6.39     & 32     & 3.31   
          & 44     & \bf\color{red}0.86    & 46            & \color{blue}\underline{1.60}            & 22             & 2.55  & 61 &1.98   &   82 & 35.60    \\
\hline
\end{tabular}\label{tab:netlib2}
\end{table}
\end{landscape}

\begin{landscape}
\begin{table}[!bp]
\ContinuedFloat
\centering
\footnotesize
\renewcommand\tabcolsep{3pt}
\caption{(cont.) Experiments on \textsc{Netlib} problems. 
In each row, red boldface and blue underline denote the fastest and second fastest solvers in CPU time, respectively.}
\renewcommand{\arraystretch}{1.1}

\begin{tabular}{@{}rrrrrrrrrrrrrrrrrrrrr@{}}
 \hline 
& && \multicolumn{2}{c}{AB-GMRES}  &  \multicolumn{2}{c}{CGNE} &  \multicolumn{2}{c}{MRNE} &  \multicolumn{2}{c}{Modified}  & \multicolumn{2}{c}{SDPT3 }  & \multicolumn{2}{c}{SeDuMi } & \multicolumn{2}{c}{MOSEK} & \multicolumn{2}{c}{PDCO}& \multicolumn{2}{c}{PDCO}\\ 

& && \multicolumn{2}{c}{(NE-SOR)}  & \multicolumn{2}{c}{(NE-SSOR)}  & \multicolumn{2}{c}{(NE-SSOR)}  & \multicolumn{2}{c}{Cholesky} & \multicolumn{2}{c}{} & \multicolumn{2}{c}{} & \multicolumn{2}{c}{} & \multicolumn{2}{c}{Direct} & \multicolumn{2}{c}{LSMR}\\
Problem & $m$  & $n$   &  Iter  & Time   & Iter & Time & Iter & Time & Iter & Time     & Iter & Time      & Iter  & Time   & Iter & Time & Iter & Time& Iter & Time \\ \hline
% name    | m     | ABGMRES          | CGNE                             | MRNE                             | Cholesky
%                 | SDPT3            | SeDuMi                           | MOSEK
% ------------------------------------------------------------------------------------------------------------------------  
pds\_\!\! 10 & 16,558 & 49,932 & 51   & 91.60   & 52            & 87.75           & 50             & 79.22           & 52        & 130.17
                  & \texttt{f}      & \texttt{f}       & 74            & 157.20          & 28             & \bf\color{red}2.54 &65&\color{blue}\underline{26.43} & 77 & 406.33\\
% ------------------------------------------------------------------------------------------------------------------------
pds\_\!\! 20 & 33,874 & 108,175 & 61   & 1,365.98 & 64            & 1,155.95         & 62             & 683.72         & 62        &  665.05      
                  & \texttt{f}     & \texttt{f}       &   $\dag_{-7}$            &   \dag            &    34            &    \bf\color{red}11.02  &86&\color{blue}\underline{147.67}  & \texttt{f} & \texttt{f}\\
\hline
\end{tabular}\label{tab:netlib22}
\end{table}
\end{landscape}

\begin{landscape}
\begin{table}[!bp]
\ContinuedFloat
\centering
\footnotesize
\renewcommand\tabcolsep{4pt}
\caption{(cont.) Experiments on \textsc{Netlib} problems.
In each row, red boldface and blue underline denote the fastest and second fastest solvers in CPU time, respectively.}
\renewcommand{\arraystretch}{1.1}
\tabcolsep=3pt
\begin{tabular}{@{}rrrrrrrrrrrrrrrrrrrrr@{}}
 \hline 
& && \multicolumn{2}{c}{AB-GMRES}  &  \multicolumn{2}{c}{CGNE} &  \multicolumn{2}{c}{MRNE} &  \multicolumn{2}{c}{Modified}  & \multicolumn{2}{c}{SDPT3 }  & \multicolumn{2}{c}{SeDuMi } & \multicolumn{2}{c}{MOSEK} & \multicolumn{2}{c}{PDCO}& \multicolumn{2}{c}{PDCO}\\ 

& && \multicolumn{2}{c}{(NE-SOR)}  & \multicolumn{2}{c}{(NE-SSOR)}  & \multicolumn{2}{c}{(NE-SSOR)}  & \multicolumn{2}{c}{Cholesky} & \multicolumn{2}{c}{} & \multicolumn{2}{c}{} & \multicolumn{2}{c}{} & \multicolumn{2}{c}{Direct} & \multicolumn{2}{c}{LSMR}\\
Problem & $m$  & $n$   &  Iter  & Time   & Iter & Time & Iter & Time & Iter & Time     & Iter & Time      & Iter  & Time   & Iter & Time & Iter & Time& Iter & Time\\ \hline
% name    | m     | ABGMRES          | CGNE                             | MRNE                             | Cholesky
%                 | SDPT3            | SeDuMi                           | MOSEK
% ------------------------------------------------------------------------------------------------------------------------  
sctap1    & 300 & 660 & 17 & \color{blue}\underline{0.31} & 19 & 0.38 & 19 & 0.36 & 17 & \bf\color{red}0.12 & 20 & 0.46    
		  & 20  & 0.50 & 11 & 2.15 &36 &0.39  & 38 & 5.06 \\
% ------------------------------------------------------------------------------------------------------------------------  
sctap2    & 1,090 & 2,500  & 20 & 1.36 & 20  & 1.21 & 21 & 1.04 & 19 & 1.75 & 21 & \bf\color{red}0.48 & 1 & \color{blue}\underline{0.60} & 9 & 2.05 & 33 & 0.85 & 35 & 6.75 \\
% -----------------------------------------------------------------------------------------------------------------------          
sctap3    & 1,480 & 3,340 & 19     & 1.33    & 19            & 1.14            & 20        & 1.22                & 18   & 2.31  
                   & 23     & \color{blue}\underline{0.94}    & 13            & \bf\color{red}0.40            & 9         & 2.11    &35 & 1.12 &   39 & 6.99 \\
\hline
\end{tabular}\label{tab:netlib3}
\end{table}
\end{landscape}

\begin{landscape}
\begin{table}[!bp]
\centering
\footnotesize
\renewcommand\tabcolsep{3pt}
\caption{Experiments on \textsc{Qaplib} problems.
In each row, red boldface and blue underline denote the fastest and second fastest solvers in CPU time, respectively.}
\renewcommand{\arraystretch}{1.1}
\begin{tabular}{@{}rrrrrrrrrrrrrrrrrrrrr@{}}
 \hline 
& && \multicolumn{2}{c}{AB-GMRES}  &  \multicolumn{2}{c}{CGNE} &  \multicolumn{2}{c}{MRNE} &  \multicolumn{2}{c}{Modified}  & \multicolumn{2}{c}{SDPT3 }  & \multicolumn{2}{c}{SeDuMi } & \multicolumn{2}{c}{MOSEK} & \multicolumn{2}{c}{PDCO}& \multicolumn{2}{c}{PDCO}\\ 

& && \multicolumn{2}{c}{(NE-SOR)}  & \multicolumn{2}{c}{(NE-SSOR)}  & \multicolumn{2}{c}{(NE-SSOR)}  & \multicolumn{2}{c}{Cholesky} & \multicolumn{2}{c}{} & \multicolumn{2}{c}{} & \multicolumn{2}{c}{} & \multicolumn{2}{c}{Direct} & \multicolumn{2}{c}{LSMR}\\

Problem & $m$  & $n$   &  Iter  & Time   & Iter & Time & Iter & Time & Iter & Time     & Iter & Time      & Iter  & Time   & Iter & Time& Iter & Time & Iter & Time\\ \hline
% name    | m     | ABGMRES          | CGNE             | MRNE            | Cholesky
%                 | SDPT3            | SeDuMi           | MOSEK
% ------------------------------------------------------------------------------------------------------------------------  
nug05   & 201   & 225  & 7      & \color{blue}\underline{0.16}   & 7     & \bf\color{red}0.09    & 7     & \bf\color{red}0.09   & 7 & 0.27 
                   & 12     & 0.36   & 5     & 0.20    & 5     & 1.81  &17 &0.48  &17 &6.60 \\
% ------------------------------------------------------------------------------------------------------------------------               
nug06   & 372   & 486  & 10     & 0.56   & 10    & 0.31    & 10    & 0.25   & 8& 0.83 
                   & 11     & \color{blue}\underline{0.22}   & 6     & \bf\color{red}0.10    & 6     & 1.84   &18 &0.65 &28&22.59 \\
% ------------------------------------------------------------------------------------------------------------------------                   
nug07   & 602   & 931  & 12     & 1.83   & 13    & 0.96    & 12    & \color{blue}\underline{0.72}    & 12 & 2.02
                   & 18     & 1.48   & 10    & \bf\color{red}0.70    & 8     & 2.09   &25 &1.98 & \texttt{f}& \texttt{f} \\
\hline
\end{tabular}\label{tab:qaplib}
\bigskip
\caption{Experiments on \textsc{Mittelmann} problems.
In each row, red boldface and blue underline denote the fastest and second fastest solvers in CPU time, respectively.}
\renewcommand{\arraystretch}{1.1}
\renewcommand\tabcolsep{1.5pt}
\begin{tabular}{@{}rrrrrrrrrrrrrrrrrrrrr@{}}
 \hline 
& && \multicolumn{2}{c}{AB-GMRES}  &  \multicolumn{2}{c}{CGNE} &  \multicolumn{2}{c}{MRNE} &  \multicolumn{2}{c}{Modified}  & \multicolumn{2}{c}{SDPT3 }  & \multicolumn{2}{c}{SeDuMi } & \multicolumn{2}{c}{MOSEK} & \multicolumn{2}{c}{PDCO}& \multicolumn{2}{c}{PDCO}\\ 

& && \multicolumn{2}{c}{(NE-SOR)}  & \multicolumn{2}{c}{(NE-SSOR)}  & \multicolumn{2}{c}{(NE-SSOR)}  & \multicolumn{2}{c}{Cholesky} & \multicolumn{2}{c}{} & \multicolumn{2}{c}{} & \multicolumn{2}{c}{} & \multicolumn{2}{c}{Direct} & \multicolumn{2}{c}{LSMR}\\
Problem & $m$  & $n$   &  Iter  & Time   & Iter & Time & Iter & Time & Iter & Time     & Iter & Time      & Iter  & Time   & Iter & Time & Iter & Time& Iter & Time\\ \hline
% name    | m     | ABGMRES          | CGNE             | MRNE                   | Cholesky
%                 | SDPT3            | SeDuMi           | MOSEK
% ------------------------------------------------------------------------------------------------------------------------  
fome11   & 12,142  & 24,460 & 47    & 6,900.09  & 48 & 14,156.31 & 53 & 12,270.84 & \texttt{f}    & \texttt{f}  
                    & \texttt{f}     & \texttt{f}        &$\dag_{-5}$ &\dag & 23   & \bf\color{red}8.97     &72&\color{blue}\underline{51.99}   &69&1,577.24 \\ 
% ------------------------------------------------------------------------------------------------------------------------       
fome12   & 24,284  & 48,920 & 48    &  12,568.26 & 48 &38,138.98  & 52   & 28,159.85    & \texttt{f}    & \texttt{f} 
                    & \texttt{f}     & \texttt{f}        &$\dag_{-7}$    &\dag & 21     & \bf{\bf\color{red}33.17}   &72 &\color{blue}\underline{101.01}  &69&3,736.93  \\
\hline
\end{tabular}\label{tab:mittle}
\end{table}
\end{landscape}
}

\begin{figure}[!htbp]
    \subfloat[Condition number $\kappa(\mathcal{A}\mathcal{A}^\mathsf{T})$.]{\includegraphics[width=.47\textwidth]{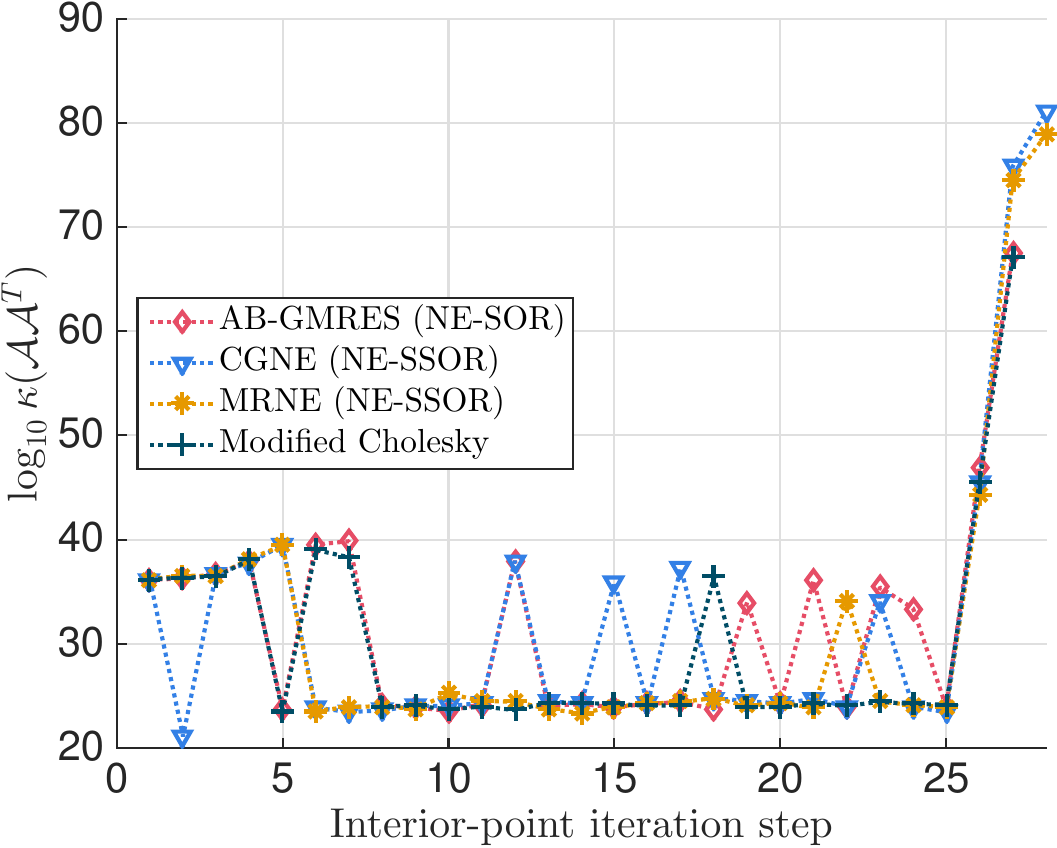}\label{fig:cond}}\hfill
    \subfloat[Duality measure $\mu$.]{\includegraphics[width=.47\textwidth]{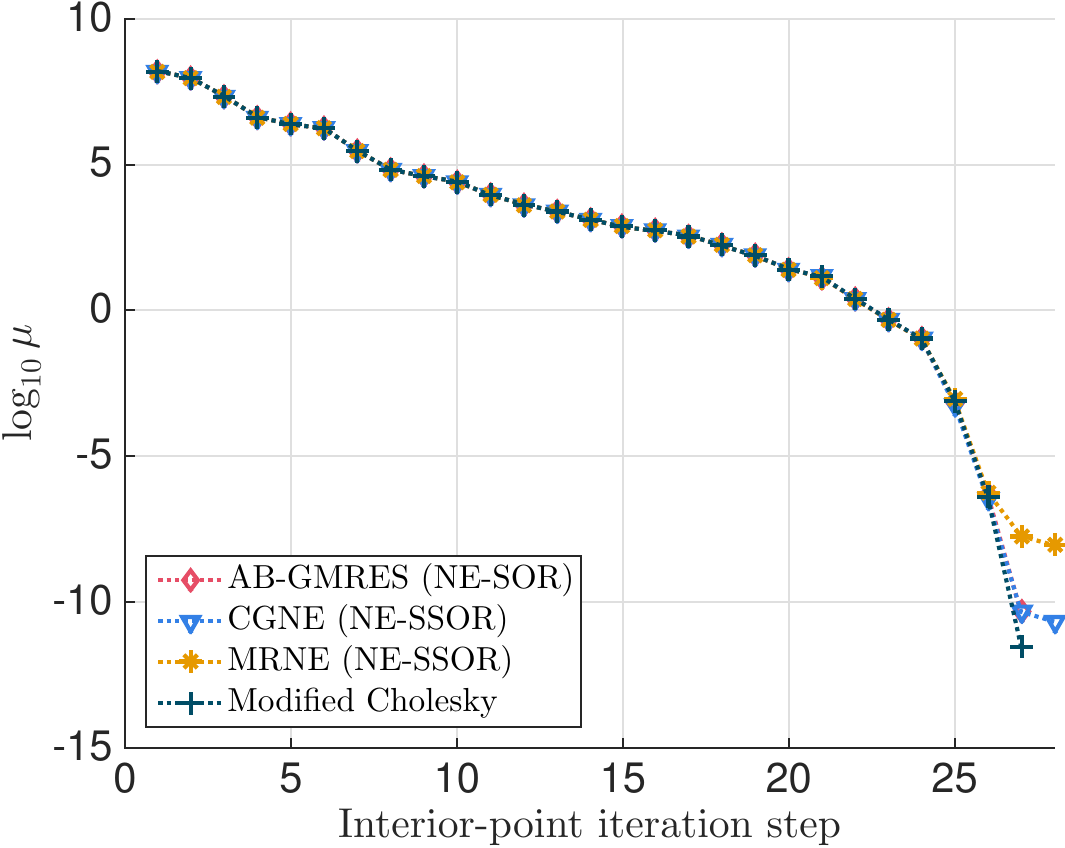}\label{fig:mu}}\\
    \subfloat[Relative residuals for predictor stage.]{\includegraphics[width=.47\textwidth]{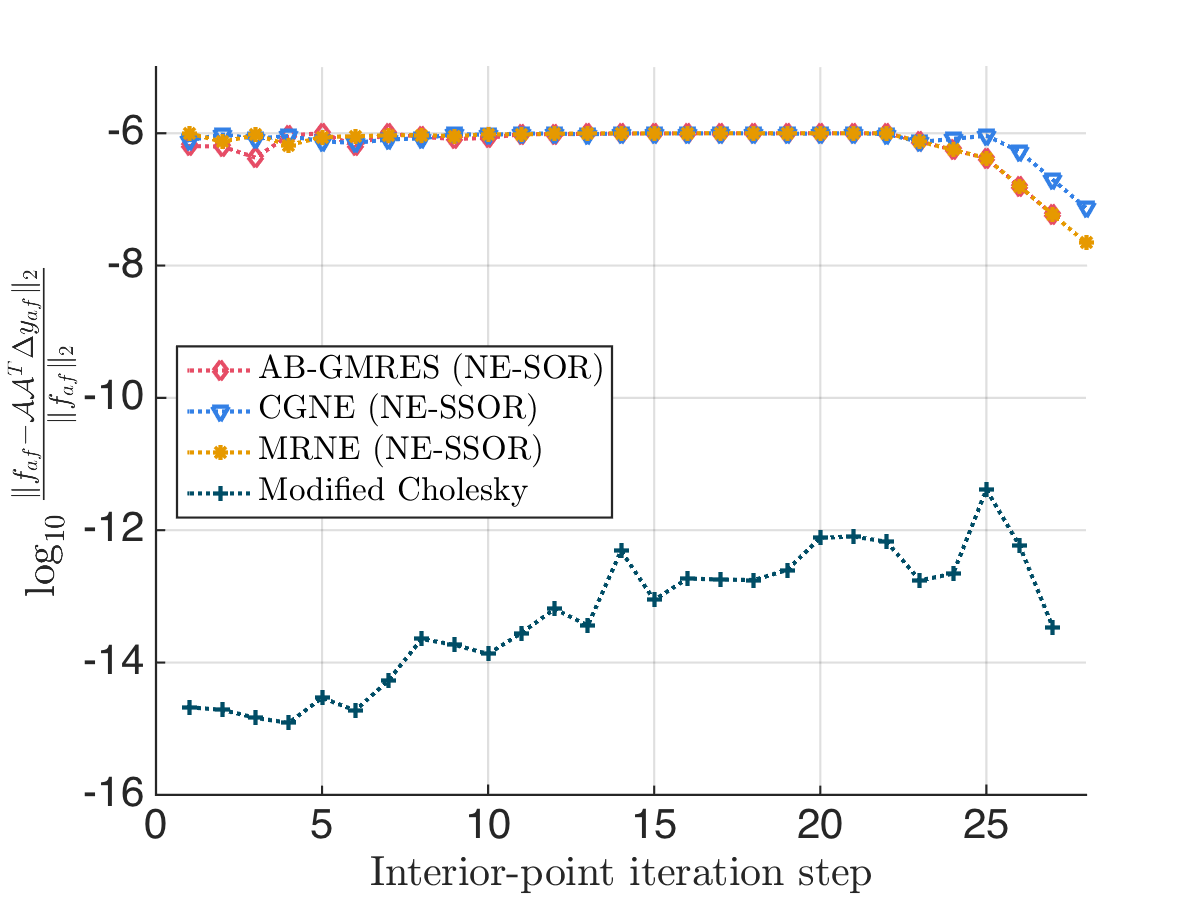}\label{fig:resmid1}}\hfill
    \subfloat[Relative residuals for corrector stage.]{\includegraphics[width=.47\textwidth]{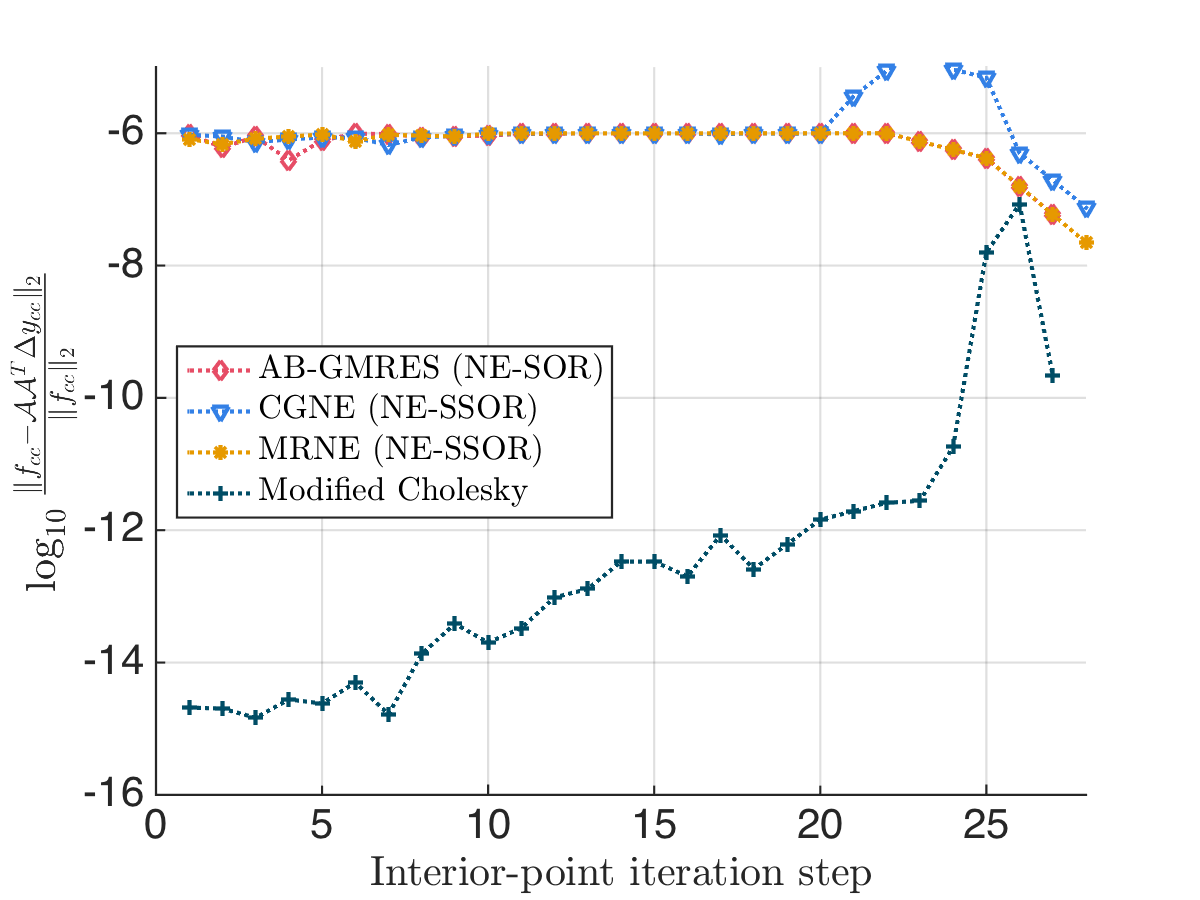}\label{fig:resmid2}}\\
    \subfloat[CPU time for each interior-point step.]{\includegraphics[width=.47\textwidth]{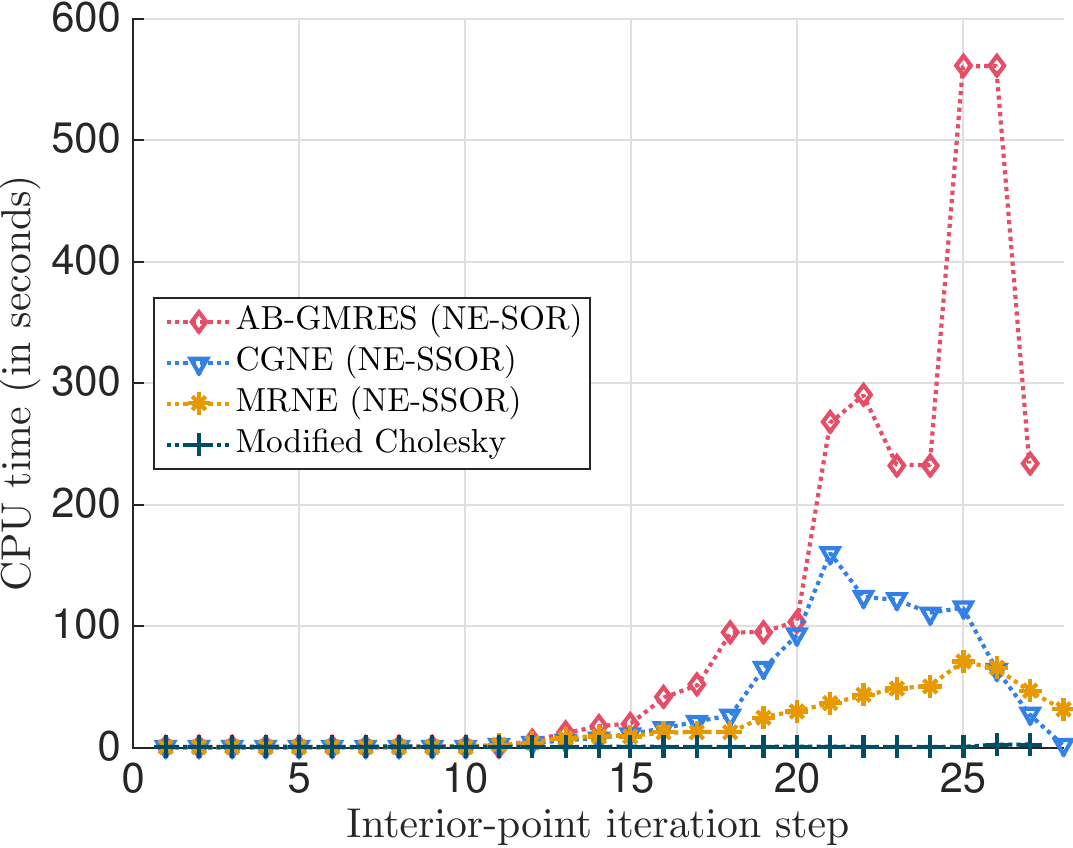}\label{fig:cputime}}\hfill
    \subfloat[Krylov iteration for each interior-point step.]{\includegraphics[width=.47\textwidth]{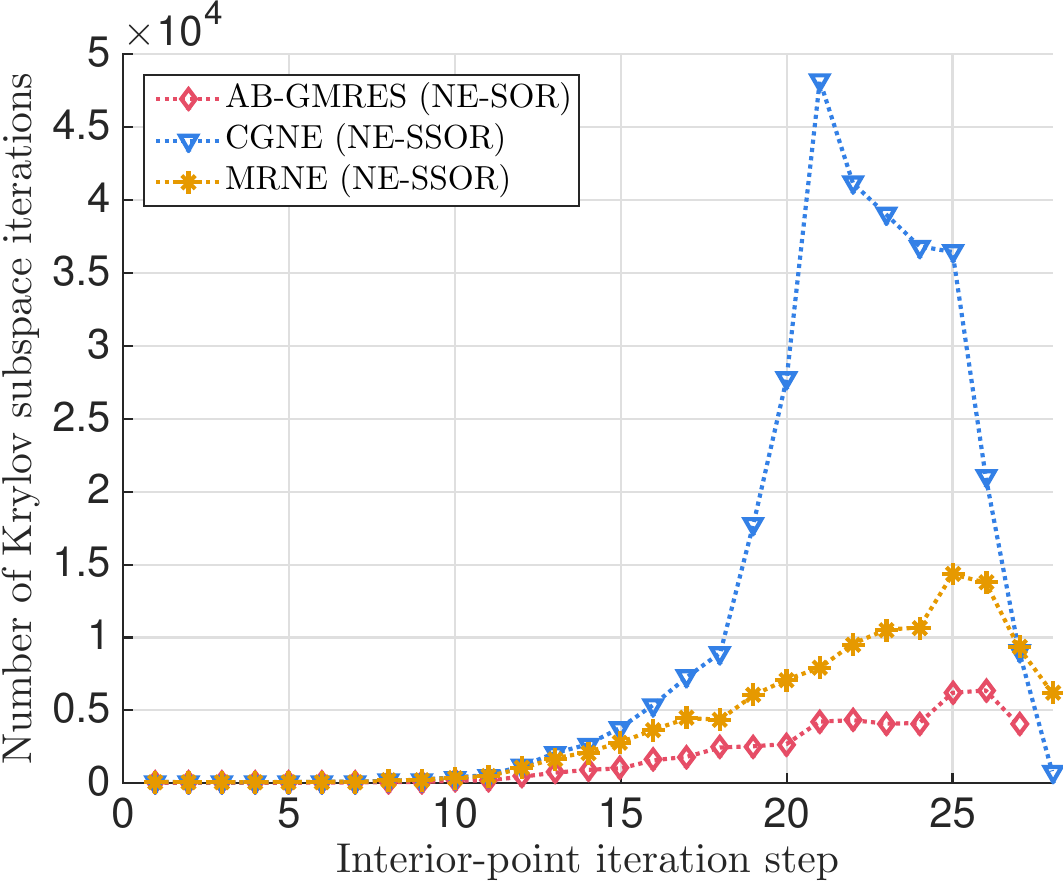}\label{fig:midit}}
    \caption{Numerical results for problem ken\_\!\! 13.\label{fig:ken13}}
\end{figure}

Figures \ref{fig:resmid1} and \ref{fig:resmid2} show the relative residual norm $\|\boldsymbol{f}_\mathrm{af} - \mathcal{A} \mathcal{A}^\mathsf{T} \Delta \boldsymbol{y}_\mathrm{af} \|_2 / \|\boldsymbol{f}_\mathrm{af}\|_2$ in the predictor stage and $\| \boldsymbol{f}_\mathrm{cc} - \mathcal{A} \mathcal{A}^\mathsf{T} \Delta \boldsymbol{y}_\mathrm{cc} \|_2/ \|\boldsymbol{f}_\mathrm{cc}\|_2$ in the corrector stage, respectively. 
The quantities are in $\log_{10}$ scale. 
The relative residual norm for modified Cholesky tended to increase with the interior-point iterations
and sharply increased in the final phase when it lost accuracy in solving the normal equations for the steps. 
We observed similar trends for other test problems and, in the worst cases, the inaccuracy in the solutions prevented interior-point convergence.
Among the iterative solvers, AB-GMRES (NE-SOR) and MRNE (NE-SSOR) were the most stable in keeping the accuracy of solutions to the normal equations; CGNE (NE-SSOR) performed similarly but lost numerical accuracy at the last few interior-point steps.

Figures \ref{fig:cputime} and \ref{fig:midit} show the CPU time and number of iterations of the Krylov methods for each interior-point step, respectively. 
It was observed that the CPU time of the modified Cholesky solver was more evenly distributed in the whole process while that of the iterative solvers tended to be less in the beginning and ending phases. 
At the final stage, AB-GMRES (NE-SOR) required the fewest number of iterations but cost much more CPU time than the other two iterative solvers. 
This can be explained as follows: AB-GMRES (NE-SOR) requires increasingly more CPU time and memory with the number of iterations because it has to store the orthonormal vectors in the modified Gram-Schmidt process as well as the Hessenberg matrix.
In contrast, CGNE (NE-SSOR) and MRNE (NE-SSOR) based methods require constant memory. 
CGNE (NE-SSOR) took more iterations and CPU time than MRNE (NE-SSOR).
Other than $\mathcal{A}$ and the preconditioner, the memory required for $k$ iterations of  AB-GMRES is 
$\mathcal{O}(k^2 + km + n)$ 
and that for CGNE and MRNE iterations is $\mathcal{O}(m+n)$
\cite{HayamiYinIto2010,MorikuniHayami2015}.
This explains why AB-GMRES (NE-SOR), although requiring fewer iterations, usually takes longer to obtain the solution at each interior-point step. We also did experiments on restarting AB-GMRES for a few problems. However, the performance was not competitive compared to the non-restarted version.

On the other hand, the motivation for using AB-GMRES (NE-SOR) is that
GMRES is more robust for ill-conditioned problems than the symmetric solvers CG and MINRES. This is because GMRES uses a modified Gram-Schmidt process to orthogonalize the vectors explicitly; CG and MINRES rely on short recurrences, where orthogonality of vectors may be lost due to rounding error.
Moreover, GMRES allows using non-symmetric preconditioning while the symmetric solvers require symmetric preconditioning. 
For example, using SOR preconditioner is cheaper than SSOR for one iteration because the latter goes forwards and backwards. SOR requires 2MV + 3m operations per inner iteration, while SSOR requires 4MV + 6m.  
In this sense, the GMRES method has more freedom for choosing preconditioners.

From Figure \ref{fig:ken13}, we may draw a few conclusions.
For most problems, the direct solver gave the most efficient result in terms of CPU time. 
However, for some problems, the direct solver tended to lose accuracy as interior-point iterations proceeded and, in the worst cases, this would inhibit convergence.
For problems where the direct method broke down, the proposed inner-iteration preconditioned Krylov subspace methods worked until convergence. 
With the iterative solvers, it is acceptable to solve \eqref{eq:SE3_predictor} and \eqref{eq:SE3_corrector} to a moderate level of accuracy in the early phase of the interior-point iterations, and then increase the level of accuracy in the late phase.

\FloatBarrier
\subsection{Basis pursuit problems}\label{sec:bp_lp}
Most of the problems tested in the last section have a sparse constraint matrix~$A$. The average nonzero density is $2.55\%$,  $0.62\%$, and $0.45\%$ for the problems in \textsc{Netlib}, \textsc{Qaplib}, and \textsc{Mittelmann}, respectively.
However, the matrix can be large and dense for problems such as QP in support vector machine training and linear programming in basis pursuit \cite{ChenDonohoSaunders1998}. 
The package Atomizer \cite{chen1995atomizer} gives such matrices.

In this section, we enrich the experiment by adding problems arising from basis pursuit \cite{ChenDonohoSaunders1998}.
We reproduced the $\ell_1$-norm optimization problems from the package Atomizer \cite{chen1995atomizer}, and reformulated them in the standard form of linear programming. The connection between basis pursuit and LP can be found therein.
The problems tested in this section have constraint matrices with average nonzero density $48.33\%$ and are usually very well-conditioned, with condition number in the range of $(1, 18.54]$.
The results are shown in Table \ref{tab:bp}.

\begin{table}[!bt]
\footnotesize
\renewcommand{\arraystretch}{1.50}
\tabcolsep=0.77pt
\caption{Experiments on basis pursuit problems.}
\begin{center}\footnotesize
\begin{tabular}{@{}rrrrrrrrrrrrrrrrrrrrr@{}}
 \hline 
& && \multicolumn{2}{c}{AB-GMRES}  &  \multicolumn{2}{c}{CGNE} &  \multicolumn{2}{c}{MRNE} & \multicolumn{2}{c}{MOSEK} & \multicolumn{2}{c}{PDCO}& \multicolumn{2}{c}{PDCO}\\ 

& && \multicolumn{2}{c}{(NE-SOR)}  & \multicolumn{2}{c}{(NE-SSOR)}  & \multicolumn{2}{c}{(NE-SSOR)} & \multicolumn{2}{c}{} & \multicolumn{2}{c}{Direct} & \multicolumn{2}{c}{LSMR}\\

Problem & $m$ & $n$  &  Iter  & Time   & Iter & Time & Iter & Time & Iter & Time     & Iter & Time & Iter & Time     \\ \hline
% name    | m     | ABGMRES          | CGNE             | MRNE                   | Cholesky
%                 | SDPT3            | SeDuMi           | MOSEK
% ------------------------------------------------------------------------------------------------------------------------  
bpfig22  & 512 & 10,240 & 8 & 141.24 & 9 & 121.68 & 9  &117.46& 6 & 686.02& \color{blue}\underline{14}& \color{blue}\underline{33.19}& \bf\color{red}14& \bf\color{red}30.50        \\ 
bpfig23  & 256 & 4,608 & \color{blue}\underline{7} & \color{blue}\underline{5.40} & 7 &10.86 & 7  &9.44&  5& 120.41& \bf\color{red}9& \bf\color{red}3.44& 79& 629.77        \\ 
bpfig24  & 256 & 2,048 & 24 &124.83  & \dag$_{-2}$ &\dag & \dag$_{-7}$  &\dag& \color{blue}\underline{18} & \color{blue}\underline{39.85} & \bf\color{red}28& \bf\color{red}13.03& 28& 42.20      \\ 
bpfig26  &1,024&14,336 & 19 & 3,138.24 &\dag$_{-4}$  &\dag & \dag$_{-5}$  &\dag& \color{blue}\underline{16} & \color{blue}\underline{1,731.40}& \bf\color{red}35& \bf\color{red}222.88& 87&14,226.49         \\ 
bpfig31  & 512 & 10,240 & 8 & 136.35 & 9 & 112.73 & 9 & 118.93  & 6 & 632.10 & \color{blue}\underline{14}& \color{blue}\underline{35.95}& \bf\color{red}14& \bf\color{red}31.61        \\ 
bpfig32  &1,024&14,336 & 20 & 2,016.11 & \dag$_{-4}$ &\dag &   \dag$_{-4}$ &\dag& \color{blue}\underline{20} & \color{blue}\underline{1,162.40}& \bf\color{red}40& \bf\color{red}227.77& \texttt{f} & \texttt{f}         \\ 
bpfig33  &1,024&22,528 & 23 & 2,507.16 & \dag$_{-5}$ &\dag &  \dag$_{-5}$ &\dag& \color{blue}\underline{26} &\color{blue}\underline{1,846.84} & \bf\color{red}41& \bf\color{red}231.69& \texttt{f} & \texttt{f}         \\ 
bpfig34  &1,024&14,336 & 20 & 1,876.94 & \dag$_{-4}$ & \dag& \dag$_{-4}$ & \dag& \color{blue}\underline{20} & \color{blue}\underline{1,174.63} & \bf\color{red}40& \bf\color{red}250.12& \texttt{f} & \texttt{f}         \\ 
bpfig41  & 256 & 4,096 & 20 & 391.63 &\dag$_{-4}$  &\dag &\dag$_{-4}$  &\dag& \color{blue}\underline{24} & \color{blue}\underline{121.53}& \bf\color{red}32& \bf\color{red}11.63&\texttt{f} &\texttt{f}         \\ 
bpfig51  &1,024&22,528 & \color{blue}\underline{20} & \color{blue}\underline{1,048.55} & \dag$_{-5}$ &\dag & \dag$_{-5}$  &\dag& 16 &1,741.09 & \bf\color{red}35& \bf\color{red}219.62& 35& 2,969.34        \\ 
bpfig52  & 256 & 2,048 & 16 & 77.93 & \dag$_{-3}$ &\dag & \dag$_{-2}$  &\dag& \color{blue}\underline{13} & \color{blue}\underline{38.39} & \bf\color{red}28& \bf\color{red}9.06&28 &105.90         \\ 
bpfig53  &1,024&4,096  & 24 & 1,447.58 &\dag$_{-4}$  &\dag &\dag$_{-4}$   &\dag & \color{blue}\underline{21} &\color{blue}\underline{156.18} & \bf\color{red}41& \bf\color{red}65.68& \texttt{f} & \texttt{f}  \\\ 
bpfig54  &1,024&4,096  & \color{blue}\underline{22} & \color{blue}\underline{1,830.62} &\dag$_{-5}$  &\dag &\dag$_{-6}$  &\dag & \bf\color{red}28& \bf\color{red}168.20&\dag$_{-6}$ & \dag&\texttt{f}&\texttt{f}\\ 
\hline
\end{tabular}\label{tab:bp}
\end{center}
\end{table}
The notations have the same meaning as explained in the previous section.
Although PDCO's direct solver may be fast for the problems in Table \ref{tab:bp}, if the problems are given without explicit constraint matrices, 
one has to use the iterative solver (e.g., LSMR) version.
The result shows that only AB-GMRES (NE-SOR) and MOSEK succeeded in solving all the problems. Among these two methods, AB-GMRES (NE-SOR) was faster than MOSEK for the problems bpfig22, bpfig23, bpfig31, and bpfig51.

\FloatBarrier
\section{Conclusions}\label{sec:lp_conclusion}
We proposed a new way of preconditioning the normal equations of the second kind arising within interior-point methods for LP problems \eqref{eq:minnrmprob}.
The resulting interior-point solver is composed of three nested iteration schemes. 
The outer-most layer is the predictor-corrector interior-point method; the middle layer is the Krylov subspace method for least squares problems, where we may use AB-GMRES, CGNE or MRNE; on top of that, we use a row-scaling scheme that does not incur extra CPU time but helps improving the condition of the system at each interior-point step; the inner-most layer, serving as a preconditioner for the middle layer, is the stationary inner iterations. 
Among the three layers, only the outer-most one runs towards the required accuracy and the other two are terminated prematurely. The linear systems are solved with a gradually tightened stopping tolerance. We also proposed a new recurrence regarding $\Delta\boldsymbol{w}$ in place of $\Delta\boldsymbol{y}$ to omit one matrix-vector product at each interior-point step. We showed that the use of inner-iteration preconditioners in combination with these techniques enables the efficient interior-point solution of wide-ranging LP problems. We also presented a fairly extensive benchmark test for several renowned solvers including direct and iterative solvers.

The advantage of our method is that it does not break down, even when the matrices become ill-conditioned or (nearly) singular.
The method is competitive for large and sparse problems and may also be well-suited to problems in which matrices are too large and dense for direct approaches to work. 
Extensive numerical experiments showed that 
our method outperforms the open-source solvers SDPT3, SeDuMi, and PDCO regarding stability and efficiency.

There are several aspects of our method that could be improved. The current implementation of the interior-point method does not use a preprocessing step except for eliminating empty rows and columns. Its efficiency may be improved by adopting some existing preprocessing procedure such as presolve to detect and remove linear dependencies of rows and columns in the constraint matrix. Also, the proposed method could be used in conjunction with more advanced interior-point frameworks such as the multi-corrector interior-point method.
In terms of the linear solver, future work is to try reorthogonalization for CG and MINRES and the Householder orthogonalization for GMRES.
It is also important to develop preconditioners that only require the action of the operator on a vector, as in huge basis pursuit problems.

It would also be worthwhile to extend our method to problems such as convex QP and SDP.

\paragraph{Acknowledgements}
We would like to thank the editor and referees for their valuable comments.

%\newpage
\bibliographystyle{siamplain}
\bibliography{ref}
\end{document}